\documentclass[12pt, reqno]{amsart}
\usepackage{amssymb, amsmath}
\usepackage{amsfonts}
\usepackage{color, dsfont}

\newtheorem{thm}{Theorem}[section]
\newtheorem{defn}[thm]{Definition}
\newtheorem{cor}[thm]{Corollary}
\newtheorem{lem}[thm]{Lemma}
\newtheorem{prop}[thm]{Proposition}
\newtheorem{hyp}[thm]{Premise}
\theoremstyle{remark}
\newtheorem{rem}[thm]{Remark}
\newtheorem{eg}[thm]{Example}

\newcommand{\norm}[1]{\left\lVert#1\right\rVert}
\newcommand{\abs}[1]{\left\lvert#1\right\rvert}
\newcommand{\inp}[1]{\left\langle#1\right\rangle}
\newcommand{\roo}[2]{|#1\rangle \langle#2|}
\newcommand{\sg}[1]{(#1(t))_{t \geq 0}}
\newcommand{\N}{\mathbb{N}}
\newcommand{\R}{\mathbb{R}}

\newcommand{\C}{\mathbb{C}}

\newcommand{\Vt}{\tilde{V}}
\newcommand{\Hb}{\mathfrak{H}}
\newcommand{\Tf}{\mathfrak{T}(\Hb)}
\newcommand{\Tsa}{\mathfrak{T}_s(\Hb)}
\newcommand{\Tc}{\mathcal{T}}
\newcommand{\SC}{\mathcal{S}}
\newcommand{\Dc}{\mathcal{D}}
\newcommand{\K}{\mathcal{K}}
\newcommand{\lop}{\mathcal{L}}
\newcommand{\loph}{\mathcal{L}(\Hb)}

\newcommand{\A}{\mathcal{A}}

\newcommand{\G}{\mathcal{G}}
\newcommand{\V}{\mathcal{V}}
\newcommand{\Img}{\operatorname{Im}}

\newcommand{\wslim}{\operatorname{w^*-lim}}
\newcommand{\Real}{\operatorname{Re}}
\newcommand{\x}{\times}
\newcommand{\ov}{\overline}
\newcommand{\ld}{\lambda}
\newcommand{\Ld}{\Lambda}
\newcommand{\ap}{\alpha}
\newcommand{\sm}{\sigma}
\newcommand{\ep}{\epsilon}
\newcommand{\One}{\mathds{1}}
\newcommand{\ab}{{\bar a}}

\newcommand{\Dl}{\Delta_\ld}
\newcommand{\BR}{BR(\ld, A)}
\newcommand{\Ups}{\Upsilon}
\newcommand{\Sp}{\operatorname{Span}}
\newcommand{\Tr}{\operatorname{Tr}}
\newcommand{\ruv}{\roo{u}{v}}
\newcommand{\St}{\tilde{S}}
\newcommand{\Gt}{\tilde{G}}

\numberwithin{equation}{section}

\begin{document}
\setcounter{section}{0}
\setcounter{subsection}{0}

\title[New Approaches to Honesty and Applications to QDS]{New Approaches to Honesty Theory and Applications in Quantum Dynamical Semigroups}
\author{Chin Pin Wong}
\address{Mathematical Institute, University of Oxford, Andrew Wiles Building, Radcliffe Observatory Quarter, Woodstock Road, Oxford, OX2 6GG, UK}
\email{wong@maths.ox.ac.uk}
%\date{20 March 2012}
\maketitle

%\makeatletter
%\vspace{-2em}
%{\centering\enddoc@text}
%\let\enddoc@text\empty % to remove the contact info from the end of the document
%\makeatother

\begin{abstract} We prove some new characterisations of honesty of the perturbed semigroup in Kato's Perturbation Theorem on abstract state spaces via three approaches, namely mean ergodicity of operators, adjoint operators and uniqueness of the perturbed semigroup. We then apply Kato's Theorem on abstract state spaces and the honesty theory linked to it to the study of quantum dynamical semigroups. We show that honesty is the natural generalisation of the notion of conservativity.
\end{abstract}

\section{Introduction}
This paper originates from a perturbation theorem for substochastic semigroups (positive semigroups which are contractive on the positive cone) which is known as Kato's Perturbation Theorem \cite{K, ALM}. The main idea in Kato's original work in \cite{K} tells us that if $A$ is the generator of a substochastic semigroup on $L^1$ and $B$ is a positive operator satisfying certain conditions, then there is an extension $G$ of $A+B$ that generates a perturbed substochastic semigroup. Although this theorem is useful as a generation result, with applications in various problems such as birth and death problems, fragmentation problems \cite{K, AB04} and transport equations \cite{A91, V87}, (see \cite[Chapters 7-10]{AB06} for a survey of the results), our interest in this theorem lies mainly in the honesty theory derived from it. 

Honesty is a property of the perturbed semigroup in Kato's Theorem. We will give the precise technical definition of honesty in Section \ref{SS-K-ASS}; for now, it suffices to think of honesty theory as the study of the consistency between the perturbed semigroup and the system it describes in the following sense: A substochastic semigroup on $L^1$ is often used to model the time evolution of some quantity. The nature of the modelled process often requires that the described quantity should be preserved, i.e.\ the semigroup describing the evolution is conservative (stochastic). However, in some cases, the semigroup turns out not to be conservative even though the modelled system should have this property. This phenomenon is what we will call \emph{dishonesty}. For a system modelled by a strictly substochastic semigroup, we have a loss term representing the loss due to the system. Dishonesty in this case would mean that the described quantity is lost from the system faster than predicted by the loss term. 

Apart from consistency with the system, honesty of the semigroup is also interesting from a purely mathematical point of view. It is a well-known result \cite[Theorem 6.13]{AB06}, \cite[Remark 1.7]{MK-V}, \cite[Theorem 2.1]{ALM} that honesty of the semigroup is equivalent to its generator $G$ being precisely equal to $\ov{A+B}$ (Kato's Theorem tells us $G \supseteq A+B$). Thus honesty characterises when a core of $A+B$ is also a core of $G$.

The honesty of the perturbed semigroup in Kato's Theorem in $L^1$ has been extensively studied, with results going back to Kato's seminal paper \cite{K} where Kato studied the stochasticity of the perturbed semigroup on $\ell^1$. Other early results include \cite{V87} and \cite{A91}. More recently, Voigt and Mokhtar-Kharroubi in \cite{MK-V}, introduced a more systematic approach to studying the problem on $L^1$, that is, via functionals involving the resolvents of generators. Note that an important area in the study of honesty involves identifying characterisations of honesty which can then be used to determine if a semigroup in a given setting is honest.  This follows naturally since the honest semigroups are the consistent semigroups. 

In this paper, we are interested in a recent generalisation of Kato's Theorem to abstract state spaces and the honesty theory linked to it. Abstract state spaces are real ordered Banach spaces with norm additive on its generating positive cone. Examples of such spaces include preduals of von Neumann algebras, or more generally, subspaces of duals of C$^*$-algebras. This generalisation was proven by Arlotti, Lods and Mokhtar-Kharroubi in \cite[Theorem 2.1]{ALM} (see also the special case for the space of trace class operators in \cite{MK08}). More significantly, in \cite[Section 3]{ALM}, the authors also showed that earlier results in the honesty theory of Kato's Theorem on $L^1$ hold for the case of abstract state spaces as well by generalising the functional approach involving resolvents. Furthermore, they introduced an alternative approach to honesty using functionals which are defined using the Dyson-Phillips series representation of the semigroup \cite[Section 4]{ALM}.

In Sections \ref{S-NewHT} and \ref{S-HonPot}, we will introduce some new approaches to the study of honesty in abstract state spaces. In particular we will present three new approaches to characterising honesty, namely a mean ergodic approach, an approach involving adjoint operators and finally, an approach involving uniqueness of the perturbed semigroup in Kato's Theorem. The approach involving adjoint operators will turn out important when considering the adjoint semigroup in applications as we will see in Section \ref{S-HTQDS} while the approach involving uniqueness of the perturbed semigroup in tied to classical questions about the uniqueness of solutions to Kolmogorov differential equations. Finally, we will also investigate the preservation of honesty under the addition of a potential term. 

In the second half of this paper, we will investigate an application of Kato's Theorem on abstract state spaces in the study of quantum dynamical semigroups. The development of Kato's original theorem in \cite{K} was inspired by the study of classical Kolmogorov differential equations on $\ell^1$, which are in turn linked to the study of stochastic processes. The non-commutative analogue of stochastic processes is linked to the study of quantum mechanics and is known as the study of quantum stochastic processes or quantum flows. The counterpart to a Markov process in the classical setting is a quantum Markov process while the corresponding semigroups are known as quantum Markov semigroups or quantum dynamical semigroups. The extension of Kato's Theorem to abstract state spaces allows us to apply this theorem to this non-commutative setting as we will demonstrate in Sections \ref{S-KQDS} and \ref{S-HTQDS}. 

The link between Kato's Theorem (the original $L^1$ version) and quantum dynamical semigroups has been known since the 1970s, when Davies in \cite{D77} showed that the techniques used in Kato's paper \cite{K} could be used to show the existence of a quantum dynamical semigroup satisfying certain conditions. 
More precisely, the application of Kato's Theorem is restricted to the special class of quantum dynamical semigroups whose generators can be represented in Lindblad form (see Definition \ref{DefLind}). Although the relation to Kato's Theorem has been noted, the actual application of Kato's Theorem in abstract state spaces to quantum dynamical semigroups has yet to be written up in the literature. We will do so in Section \ref{S-KQDS} as this setting will be utilised in Section \ref{S-HTQDS}. However, it should be noted that the theory of quantum dynamical semigroups has largely developed independently of Kato's Theorem, see for example \cite{C99, FC, FM, F99}.

The application of Kato's Theorem in constructing quantum dynamical semigroups leads naturally to questions about the role of honesty for these semigroups. Note that although the relation with Kato's Theorem has long been identified, the link to honesty has yet to be established as honesty theory was developed in a systematic manner much more recently. However, just like the commutative case on $L^1$, there is also a notion of conservativity of quantum dynamical semigroups that has been studied independently of honesty theory (see \cite{C99, FC, F99} for example) and we will in fact, show that honesty is a generalisation of this notion in Section \ref{S-HTQDS}. More precisely, we will show that previously known results for the conservative case regarding domain characterisation and uniqueness of the semigroup can be extended to the substochastic case via honesty theory. This abstract approach to conservativity yields sleeker proofs and new insights into quantum dynamical semigroups.

\section{Kato's Theorem and Honesty Theory in Abstract State Spaces}\label{SS-K-ASS}
Let us begin by clarifying some terminology. Suppose $X$ is an ordered Banach space with positive cone $X_+$. We will often use $\inp{\cdot,\cdot}$ to denote the duality between $X$ and its dual space $X^*$. If $T$ is a linear operator on $X$, we will use $T^*$ to denote its adjoint (dual) operator. We say that the linear operator $T$ is positive if $Tu \in X_+$ for all $u \in X_+$. $T$ will be called substochastic (resp.\ stochastic) if $T$ is positive and $\norm{Tx} \leq \norm{x}$ (resp.\ $\norm{Tx} = \norm{x}$) for all $x \in X_+$. A one-parameter semigroup of operators in $X$ will be called substochastic (resp.\ stochastic) if $U(t)$ is substochastic (resp.\ stochastic) for all $t \geq 0$.

In this section, we will introduce Kato's Theorem and the honesty theory related to it. We will be interested solely in the theory of honesty of Kato's Perturbation Theorem in abstract state spaces.

An abstract state space is a real ordered Banach space, $X$, with a generating positive cone, $X_+$, on which the norm is additive, i.e. $\norm{u+v} = \norm{u} + \norm{v}$ for all $u,v\in X_+$. The additivity of the norm ensures that the norm is monotone i.e.\ $\norm{u} \leq \norm{v}$ if $0 \leq u \leq v$. Moreover, the additivity of the norm and the generating cone allow us to extend the norm on the positive cone to a linear functional on $X$ given by
\begin{equation}\label{E-Psi}
\Psi:X \to \R, \quad \inp{\Psi,u} = \norm{u},\quad u \in X_+.
\end{equation}  
For these spaces, Arlotti et al.\ proved the following generalisation of Kato's Perturbation Theorem. 
\begin{thm}\label{K-Ass} \cite[Theorem 2.1]{ALM} Suppose $X$ is an abstract state space and the operators $A$ and $B$ with $D(A) \subseteq D(B) \subseteq X$ satisfy:
\begin{enumerate}
\item $A$ generates a substochastic semigroup $(U_A(t))_{t \geq 0}$,
\item $Bu \geq 0$ for $u \in D(A)_+$,
\item\label{E-MCond} $\inp{\Psi, (A+B)u} \leq 0$ for all $u \in D(A)_+ = D(A) \cap X_+$.
\end{enumerate}
Then there exists an extension $G$ of $A+B$ that generates a substochastic $C_0$-semigroup $(V(t))_{t \geq 0}$ on $X$. The generator $G$ satisfies, for all $\ld>0$ and $u \in X$,
\begin{equation*}%\label{E-ResSeries}
R(\ld,G)u = R(\ld,A)\sum_{k=0}^\infty (BR(\ld,A))^ku.
\end{equation*}

Moreover, $(V(t))_{t \geq 0}$ is the minimal substochastic $C_0$-semigroup whose generator is an extension of $(A+B)$.
\end{thm}
Henceforth, we will refer to Theorem \ref{K-Ass} as Kato's Theorem.

We will discuss two approaches to honesty theory of Kato's Theorem as they will be required in the subsequent sections. We begin by introducing the functional approach involving resolvent operators from \cite[Section 3]{ALM}. Consider again the operators from Theorem \ref{K-Ass}. We are interested in the functional \[a_0:D(G) \to \R, \qquad a_0(u)=-\inp{\Psi, Gu}.\] Since $\sg{V}$ is substochastic, for $u\in D(G)_+$, we have \[\inp{\Psi, Gu}= \lim_{t \to 0} t^{-1}\inp{\Psi, V(t)u-u} =\lim_{t \to 0} t^{-1}(\norm{V(t)u}-\norm{u}) \leq 0,\] so $a_0$ is positive on $D(G)$. Moreover, from the definition, it is easy to see that $a_0$ is continuous on $D(G)$ with respect to the graph norm. We denote the restriction of $a_0$ to $D(A)$ by $a$, i.e.
\begin{equation}\label{E-ab}
a_0|_{D(A)} = a:D(A) \to \R, \qquad a(u) = -\inp{\Psi, Au+Bu}.
\end{equation}       

We now use $a$ to define our second functional. Fix $\ld>0$ and $u \in X_+$. Since $R(\ld, A)$ and $BR(\ld,A)$ are positive operators, the sequence $R^{(n)}u:= \sum_{k=0}^n R(\ld, A)(BR(\ld, A))^k u$, $n \in \N$ is non-decreasing and in fact, converges to $R(\ld, G)u$. Therefore, we have $a(R^{(n)} u) = a_0(R^{(n)} u) \leq a_0(R(\ld, G)u)$ for all $n \in \N$, i.e.\ $(a(R^{(n)} u))_n$ is a bounded, monotone real sequence, which must then be convergent. Taking $u=u_+ -u_- \in X$, $u_+, u_- \in X_+$, we see that this convergence holds for any $u \in X$. Therefore, we can define a new functional on $D(G)$ by
\[\ab_\ld(R(\ld, G)u) = \sum_{k=0}^\infty a(R(\ld, A)(BR(\ld, A))^k u).\]

It can be shown \cite[Proposition 3.1]{ALM} that $\ab_\ld|_{D(A)} = a$ and that the definition of $\ab_\ld$ is independent of $\ld$. Thus we define $\ab:= \ab_\ld$. From the inequality $a(R^{(n)} u) \leq a_0(R(\ld, G)u)$ for $u \in X_+$, it follows that $\ab(R(\ld,G)u) \leq  a_0(R(\ld,G)u)$. Hence, $\ab$ is continuous on $D(G)$ with respect to the graph norm.

These two functionals now allow us to define a positive functional, $\Dl \in X^*$ which will be key in characterising the honesty of the semigroup,
\begin{equation}\label{E-Dl}
\inp{\Dl,u} = a_0(R(\ld, G)u)-\ab(R(\ld, G)u),\quad u \in X.
\end{equation}

To see this, we need the technical definition of honesty as given in \cite{ALM}. To motivate the definition, consider the following: For any $u \in X_+$ and any $t \geq 0$, we have $\int_0^t V(s)u\,ds \in D(G)$ with
$V(t)u-u=G\int_0^t V(s)u\,ds.$
Since the semigroup is positive, we have
\begin{equation}\label{E-SgGen}
\norm{V(t)u}-\norm{u} = -a_0\left(\int_0^t V(s)u\,ds\right).
\end{equation}

We define honesty to be the following:
\begin{defn}\label{Def-HAss}\cite[Definition 3.8]{ALM} The perturbed semigroup $\sg{V}$ in Kato's Theorem is said to be honest if and only if 
\begin{equation}\label{E-SgHon}
\norm{V(t)u}-\norm{u} = -\ab\left(\int_0^t V(s)u\,ds\right) \qquad \text{ for all }t \geq 0, u \in X_+.
\end{equation}
Otherwise, the semigroup is said to be dishonest.
\end{defn}
\begin{rem}\label{HonStoc} Note that if equality holds in condition (iii) in Kato's Theorem, then $\ab=0$. Hence an honest semigroup in this case is simply a stochastic semigroup.
\end{rem}

Comparing \eqref{E-SgGen} and \eqref{E-SgHon}, we see that $(V(t))_{t \geq 0}$ is honest if and only if 
\begin{equation}\label{E-aabHon}
a_0\left(\int_0^t V(s)u\,ds\right) = \ab\left(\int_0^t V(s)u\,ds\right)\qquad \text{ for all }t \geq 0, u \in X_+.
\end{equation}
Further calculations (see for example \cite[Theorem 3.11]{ALM}) show that \eqref{E-aabHon} holds if and only if $a_0(R(\ld, G)u) =\ab(R(\ld, G)u)$ for some $\ld>0$. Therefore $\sg{V}$ is honest if and only if $\Dl = 0$, i.e. no loss occurs. This is precisely the equivalence of (i) and (ii) in Theorem \ref{HT-R}, which states some well-known characterisations of honesty. The result as stated below, can be derived from \cite[Theorem 3.5]{ALM} and Definition \ref{Def-HAss}.  
\begin{thm}\label{HT-R} Suppose $X$ is an abstract state space and $A, B, \sg{V}$ are as in Kato's Theorem and $\ld >0$. The following are equivalent.
\begin{enumerate}
\item $\sg{V}$ is honest. 
\item $\Dl=0$. 
\item $\lim_{n \to \infty}\norm{[BR(\ld, A)]^n u}=0$ for all $u \in X_+$.
\item $G=\ov{A+B}$.
\end{enumerate}
\end{thm}
The second approach to honesty we will utilise is a spectral approach. The spectral characterisation of honesty is based on Theorem \ref{HTSpect} which is a general result that is not merely restricted to the operators $A,B$ satisfying the conditions of Kato's Theorem.
\begin{thm}\label{HTSpect}\cite[Theorem 4.3]{AB06} Let $A,B$ be linear operators on a Banach space $X$ with $D(A) \subseteq D(B)$. Suppose $A+B$ has a %{\color{red} closed (invertible means closed)} 
extension $G$ and $\Ld:=\rho(A) \cap \rho(G) \neq \emptyset$. Then
\begin{enumerate}
\item $1 \notin \sm_p(BR(\ld, A))$ for any $\ld \in \Ld$.
\item $1 \in \rho(BR(\ld, A))$ for some/all $\ld \in \Ld$ if and only if $D(G) = D(A)$ and $G = A+B$.
\item $1 \in \sm_c(BR(\ld, A))$ for some/all $\ld \in \Ld$ if and only if $D(G) \supsetneq D(A)$ and $G = \ov{A+B}$.
\item $1 \in \sm_r(BR(\ld, A))$ for some/all $\ld \in \Ld$ if and only if $G \supsetneq \ov{A+B}$.
\end{enumerate}
\end{thm}
We can use Theorem \ref{HTSpect} to obtain some characterisations of honest semigroups. Henceforth, we will let $A, B$ denote the operators in Kato's Theorem unless stated otherwise. We need a simple but useful observation which follows from the equality $(\ld-A-B)R(\ld, A) = I-\BR$ and the fact that $\ker(T^*)$ is the annihilator of $\Img(T)$ for a densely defined linear operator $T$.
\begin{lem}\label{BddImg}\cite[Lemma 3.1]{TK} For all $\ld>0$, $\Img(\ld-A-B) = \Img(I-\BR)$ and $\ker(\ld-(A+B)^*)= \ker(I-(\BR)^*)$.
\end{lem}
From Theorems \ref{HT-R} and \ref{HTSpect}, we have that $\sg{V}$ is honest if and only if $1 \in \rho(BR(\ld, A)) \cup \sm_c(BR(\ld, A))$  for some $\ld>0$, which holds if and only if $\Img(I-\BR)$ is dense in $X$. Combining this with Lemma \ref{BddImg}, we have:

\begin{prop}\label{H-dense} Let $X$ be an abstract state space and $A,B$, $\sg{V}$ be as in Theorem \ref{K-Ass}. The semigroup $\sg{V}$ is honest if and only if either of the following equivalent conditions hold:
\begin{enumerate}
\item $\Img(\ld-(A+B))=\Img(I-\BR)$ is dense in $X$ for some/all $\ld > 0$. 
\item $\ker(\ld-(A+B)^*) =\ker(I-(\BR)^*) = \{0 \}$ for some/all $\ld>0$.	
\end{enumerate}
\end{prop}

%{\color{red} complexification and honesty complexified, both spectral case and functional case}

It is a well-known fact that when working with positive operators, it is generally sufficient to work in real spaces (see for example \cite[Example 2.87]{AB06}) as the theory on real spaces can be extended to complex spaces via the process of complexification. This statement applies to Kato's Theorem and honesty theory as well. However, it is sometimes more useful to apply the complex versions as we will see in Sections \ref{S-KQDS} and \ref{S-HTQDS}. So let us briefly elaborate on this. In the rest of this section, we let $X$ be an abstract state space and $A$, $B$, $G$ and $\sg{V}$ be as in Theorem \ref{K-Ass}. We will use $X_C$ to denote the complexification of $X$ %with norm $\norm{x+iy}_C = \sup_{0 \leq \theta \leq 2\pi} \norm{(\cos \theta) x+ (\sin \theta) y}, x, y \in X$ 
and  $A_C$, $B_C$, $G_C$ and $\sg{V_C}$ to denote the respective complexified operators.
% Note that the definition of the complexification of operators indicates that $G_C$ is the generator of $\sg{V_C}$.
%First we show that we have a definition of complex honesty using functionals which agrees with that of honesty on real spaces. We apply the complexified forms of the functionals $a_0$ and $\ab$. Recall from Definition \ref{Def-HAss} that for abstract state spaces, the semigroup $\sg{V}$ is honest if and only if
%\begin{equation*}
%\norm{V(t)u}-\norm{u} = -\ab\left(\int_0^t V(s)u\,ds\right) \qquad \text{ for all }u \in X_+, t \geq 0.
%\end{equation*}
%By observing that $(X_C)_+ = X_+$, $\norm{v}_C = \norm{v}$ for all $v \in (X_C)_+$ and $\ab$ and $\sg{V}$ are positive operators, it follows that this is
%equivalent to 
%\begin{equation*}
%\norm{V_C(t)u}_C-\norm{u}_C = -\ab_C\left(\int_0^t V_C(s)u\,ds\right) \qquad \text{ for all }u \in (X_C)_+, t \geq 0.
%\end{equation*}
By using the complexified forms of the functionals $a_0$ and $\ab$, we can define honesty in the complexification of abstract state spaces in terms of complex functionals:
\begin{defn}\label{Def-HAssC} The semigroup $(V_C(t))_{t \geq 0}$ is said to be honest if and only if 
\begin{equation*}
\norm{V_C(t)u}_C-\norm{u}_C = -\ab_C\left(\int_0^t V_C(s)u\,ds\right) \qquad \text{ for all }t \geq 0, \text{all } u \in (X_C)_+.
\end{equation*}
Otherwise, the semigroup is said to be dishonest.
\end{defn}
Using this definition of complex honesty, it can be shown that real honesty is equivalent to complex honesty. Moreover, we can derive analogous characterisations for complex honesty from real honesty, for example, the conditions in Proposition \ref{H-dense} for the complex case would then hold for all $\ld \in \C_+$ where $\C_+:=\{\ld \in \C \::\: \Real(\ld) >0\}$.

\section{New Characterisations of Honesty Theory}\label{S-NewHT}
We will prove some new characterisations of honesty in abstract state spaces. %{\color{red} Although we will prove each result directly in this section, we will see later in Section \ref{S-Unify} that some of the results can be derived from a more general theory.}

\subsection{Honesty and Mean Ergodicity}\label{SS-HTME}
We begin by applying the Mean Ergodic Theorem to obtain a characterisation of honesty. The advantage of this approach is that it allows us to characterise not only when the semigroup is honest but also characterise the exact form of the generator when it is honest. More precisely, we can find conditions which differentiate when the generator $G=A+B$ and when $G= \ov{A+B}$. 

%As we saw in Theorem \ref{HT-R}, the operator $\BR$ plays an important role in honesty. So we begin with an auxiliary lemma about the properties of $\BR$ which will turn up again repeatedly later (for example in Section \ref{S-Kbeyond}, Proposition \ref{Eg}, Theorem \ref{HonASSUni}).
%We begin with some preliminary information about mean ergodic and uniformly ergodic operators. As a corollary of the Mean Ergodic Theorem \cite[Theorem 2.3.1]{Kr} we have: 
%\begin{lem}\label{MET}\cite[Corollary 2.2]{TK} Let $T$ be a bounded operator on a Banach space $X$. Then $T$ is mean ergodic and $\ker(I-T) = \{0\}$ if and only if $T$ is Ces\'{a}ro bounded, $\lim_{n\to \infty}\frac{1}{n}\norm{T^n u}=0$ for all $u \in X$ and $X = \ov{\Img(I-T)}$.
%\end{lem}
%
%A similar result holds for uniformly ergodic operators and may be derived from the Uniform Ergodic Theorem and its proof \cite[Theorem 2.2.1]{Kr}. 
%\begin{lem}\label{UET}\cite[Theorem 2.4]{TK} 
%Let $T$ be a bounded linear operator on a Banach space $X$. Then $T$ is uniformly ergodic if and only if $\lim_{n\to \infty}\frac{1}{n}\norm{T^n}=0$ and $\Img(I -T)$ is
%closed.
%\end{lem}

We begin by noting that the operator $\BR$ is substochastic and hence, power bounded. This follows from the inequality $\inp{\Psi, (A+B)R(\ld,A)u} =-\norm{u}+\norm{BR(\ld,A)u} +\ld \norm{ R(\ld,A)u} \leq 0$ for $u \in X_+$. Now we can prove the following characterisation of honesty:
\begin{thm}\label{HT-ME}
 Let $X$ be an abstract state space and suppose $A,B$, $\sg{V}$ are as in Theorem \ref{K-Ass}. Then the semigroup $\sg{V}$ is honest if and only if $\BR$ is mean
ergodic for some $\ld>0$. Moreover, the generator $G=A+B$ if and only if $\BR$ is uniformly ergodic. 
\end{thm}
\begin{proof}
Fix $\ld>0$. Since $\BR$ is power-bounded, $\BR$ is Ces\'{a}ro bounded and $\lim_{n\to \infty}\frac{1}{n}\norm{(\BR)^n u}=0$ for all $u \in X$. Moreover, Theorem \ref{HTSpect} tells us that $\ker(I-\BR) = \{0\}$. Hence by the Mean Ergodic Theorem \cite[Theorem 2.1.3]{Kr}, the mean ergodicity of $\BR$ is equivalent to the condition $X = \ov{\Img(I-\BR)}$. Applying Proposition \ref{H-dense}, it follows that $\BR$ is mean ergodic if and only if $\sg{V}$ is honest. 

To prove the second assertion, we use the Uniform Ergodic Theorem. Since $\BR$ is power bounded, it satisfies the condition $\lim_{n\to
\infty}\frac{1}{n}\norm{(\BR)^n}=0$. Hence by the Uniform Ergodic Theorem \cite[Theorem 2.2.1]{Kr} and the fact that $\ker(I-\BR) =\{0\}$, it follows that $\BR$ is uniformly ergodic if and only
if $\Img(I-\BR) = X$. But $\Img(I-\BR) = X$ if and only if $I-\BR$ is invertible (as $\ker(I-\BR) =\{0\}$). Hence by Theorem \ref{HTSpect}, it follows that
$G=A+B$ if and only if $\BR$ is uniformly ergodic.
\end{proof}

The mean ergodic characterisation given in Theorem \ref{HT-ME} is motivated by some results of Tyran-Kami\'{n}ska in \cite{TK} and \cite{TK09}. In \cite[Theorem 1.3]{TK} Tyran-Kami\'{n}ska proves a generation theorem for an honest semigroup in a variant of Kato's Theorem in real Banach lattices under the additional condition that $\BR$ is mean ergodic. In \cite[Theorem 3.4]{TK09} on the other hand, she proves that the perturbed semigroup in Kato's original theorem in $L^1$ (satisfying condition (iii) in Theorem \ref{K-Ass} with equality) is stochastic if and only if $\BR$ is mean ergodic.

\subsection{Honesty and Adjoint Operators}
Next, we give a characterisation for honesty based on adjoint operators. The spectral characterisation in Theorem \ref{HTSpect} and Proposition \ref{H-dense} showed that honesty of the semigroup is related to the existence of eigenvectors of $(\BR)^*$ and $(A+B)^*$. The following result tells us that it suffices to consider the existence of positive subeigenvectors of $(\BR)^*$ and $(A+B)^*$ instead of eigenvectors.
\begin{thm}\label{HT-dual}
Let $X$ be an abstract state space and suppose $A,B$, $\sg{V}$ are as in Theorem \ref{K-Ass}. Fix $\ld>0$. The following are equivalent:
\begin{enumerate}
 \item $\sg{V}$ is dishonest.
\item There exists $f_\ld \in X^*\backslash\{0\}$ such that $((\BR)^{*n} f_\ld)$ is weak$^*$-convergent and $\wslim_{n \to \infty} (\BR)^{*n} f_\ld \neq 0$.
\item There exists $f_\ld \in X_+^*\backslash\{0\}$ such that $((\BR)^{*n} f_\ld)$ is weak$^*$-convergent and $\wslim_{n \to \infty} (\BR)^{*n} f_\ld \neq 0$.
\item There exists $f_\ld \in X_+^*\backslash\{0\}$ such that $(\BR)^{*} f_\ld \geq f_\ld$.
\item There exists $f_\ld \in X_+^*\backslash\{0\}$ such that $(A+B)^{*} f_\ld \geq \ld f_\ld$.
\item There exists $f_\ld \in X_+^*\backslash\{0\}$ such that $(A+B)^{*} f_\ld = \ld f_\ld$.
\item There exists $f_\ld \in X^*\backslash\{0\}$ such that $(A+B)^{*} f_\ld = \ld f_\ld$.
\end{enumerate}
\end{thm}
\begin{proof}
(i) $\Leftrightarrow$ (vii) follows directly from Proposition \ref{H-dense} (ii). (i) $\Rightarrow$ (vi) because $\Dl \in \ker(I-(\BR)^*)=\ker(\ld-(A+B)^*)$ \cite[Theorem 3.24]{ALM} and dishonesty of the semigroup implies $\Dl \neq 0$. (vi) $\Rightarrow$ (v) is obvious. To show (v) $\Rightarrow$ (iv), note that
$\BR -I = (A+B - \ld) R(\ld,A)$. Hence $(\BR)^* -I \supseteq R(\ld, A^*)((A+B)^*)-\ld)$. Since $R(\ld, A^*)$ is a positive operator, it follows that
(v) $\Rightarrow$ (iv). 

To show (iv) $\Rightarrow$ (iii), we first note that $(\BR)^*$ is a power bounded operator as $\BR$ is. Moreover, (iv) and the
positivity of $(\BR)^*$ implies that for all $u \in X_+$, $\inp{ (\BR)^{*n} f_\ld, u } \geq \inp{ (\BR)^{*(n-1)} f_\ld, u }$ for all $n \in \N$. Hence $\{ \inp{(\BR)^{*n} f_\ld, u } \}$ is a monotonically increasing sequence in $\R$ which is bounded above, hence converges. Since the sequence is bounded below by
$\inp{f_\ld, u}$ and $f_\ld \neq 0$, it converges to a non-zero element. Since this holds for all $u \in X_+$ and $X_+$ is generating, $\wslim (\BR)^{*n} f_\ld$ exists and is non-zero.  

(iii) $\Rightarrow$ (ii) is clear. It remains to show (ii) $\Rightarrow$ (i). Let $F:= \wslim (\BR)^{*n} f_\ld$ and $F_n:= (\BR)^{*n} f_\ld$, $n \in \N$. Take $u \in X$.
Then \[\inp{F_n, u} = \inp{(\BR)^*F_{n-1}, u} = \inp{F_{n-1}, \BR u }.\] Letting $n \to \infty$ on both sides, we have $\inp{F, u} = \inp{F, \BR u } =
\inp{(\BR)^*F, u}$, i.e. $ (I-(\BR)^*)F =0$. Thus by Proposition \ref{H-dense} (ii), $\sg{V}$ is dishonest.
\end{proof}

\subsection{Honesty and Uniqueness of the Kato Semigroup}\label{S-HTUniq} 
The final characterisation of honesty that we present is based on the uniqueness of the perturbed semigroup in Kato's Theorem. Recall from Theorem \ref{K-Ass} that the generator $G$ of the perturbed semigroup in Kato's Theorem is an extension of $A+B$. It turns out that the Kato semigroup is the unique substochastic semigroup generated by an extension of $A+B$ if and only if the Kato semigroup is honest. More precisely,
\begin{thm}\label{H-UAB} Let $X$ be an abstract state space and suppose $A,B$ satisfy the conditions of Theorem \ref{K-Ass} with perturbed semigroup $\sg{V}$ and generator $G$.  
\begin{enumerate} 
\item If $\sg{V}$ is honest, then $\sg{V}$ is the unique substochastic semigroup whose generator is an extension of $A+B$.
\item If $\sg{V}$ is dishonest, then there are infinitely many substochastic semigroups whose generators are extensions of $A+B$. 
\end{enumerate}
\end{thm}
We will require a few generation results for substochastic semigroups in order to prove Theorem \ref{H-UAB}. The first result is a version of the Hille-Yosida Theorem for substochastic semigroups which can be proven in essentially the same way, only with additional positivity constraints. 
\begin{lem}\label{HYSubst} Let $X$ be an ordered Banach space. An operator $A$ on $X$ with dense domain generates a substochastic (resp.\ stochastic) semigroup if and only if for every $\ld>0$, $A$ has a resolvent $R(\ld, A)$ with domain $X$ and $\ld R(\ld, A)$ is a substochastic (resp.\ stochastic) operator. 
\end{lem}

Using this lemma we can derive another generation result for substochastic semigroups on abstract state spaces.
\begin{lem}\label{SubstGen} Let $X$ be an abstract state space. A linear operator $A$ with dense domain generates a substochastic (resp.\ stochastic) semigroup on $X$ if and only if 
\begin{enumerate}
\item $\inp{\Psi, Au} \leq 0 \:(=0)$ for all $u \in D(A)_+$, and
\item for each $\ld>0$ and $u \in X$, the equation 
\begin{equation}\label{E-ResSoln}
\ld v-Av =u
\end{equation}
has a unique solution $v =R(\ld, A)u \in D(A)$ and $R(\ld, A)u \in X_+$ for all $u \in X_+$. 
\end{enumerate} 
\end{lem}
\begin{proof} The necessity follows directly from Lemma \ref{HYSubst} and the fact that that if $A$ generates a substochastic semigroup $\sg{T}$, then for
$u \in D(A)_+$, \[\inp{\Psi, Au} = \lim_{t \to 0} \inp{\Psi,\frac{T(t)u-u}{t}} = \lim_{t \to 0}\frac{1}{t} (\norm{T(t)u}-\norm{u}) \leq 0.\] The sufficiency follows since for
$u \in X_+$ and $\ld>0$, we have $\inp{\Psi, \ld R(\ld, A)u} = \inp{\Psi, u} + \inp{\Psi, A R(\ld, A)u} \leq \inp{\Psi, u}$ i.e. $\norm{\ld R(\ld, A)u} \leq
\norm{u}$. Hence by Lemma \ref{HYSubst}, $A$ generates a substochastic semigroup.
\end{proof}
%
%We will also need the following simple fact in our proof of Theorem \ref{H-UAB}:
%\begin{lem}\label{SplDomEq} Let $T$ be the generator of a $C_0$-semigroup (or its adjoint) on a Banach space $X$ and $S$ be a closed extension of $T$. If $S$ also generates a $C_0$-semigroup (resp. its adjoint), then $S=T$. 
%\end{lem}

\begin{proof}[Proof of Theorem \ref{H-UAB}] To prove (i), let $\sg{\Vt}$ be another substochastic semigroup with generator $\Gt \supset A+B$. Then $(A+B)^*
\supset \Gt^*$. Since $\sg{V}$ is honest, $\ov{A+B} = G$, i.e. $(A+B)^* = G^*$ and so $(A+B)^*$ generates an adjoint semigroup. Since $(A+B)^*$ and $\Gt^*$ both generate adjoint semigroups, it follows that $\Gt^* = (A+B)^* = G^*$. Thus for $\ld>0$, $R(\ld, \Gt)^* =R( \ld, G)^*$ and so $R(\ld, G) = R(\ld, \tilde{G})$. Therefore, by the Post-Widder Inversion Formula, $\tilde{V}(t) = V(t)$ for all $t \geq 0$, so $\sg{V}$ is unique.

To prove (ii), we will construct an infinite set of substochastic semigroups whose generators are extensions of $A+B$. Fix $u_0 \in X_+\backslash \{0\}$ such
that $\norm{u_0} \leq 1$. Define $\Gt$ by 
\begin{equation}\label{E-GNonUniq}
\Gt u = Gu +(a_0-\ab)(u)u_0, \qquad u \in D(\Gt).
\end{equation}
Then $\Gt$ has domain $D(G)$ and for $u \in D(A+B)=D(A)$, we
have $\Gt u = Gu = (A+B)u$ since  $a_0|_{D(A)} = a= \ab|_{D(A)}$. Hence, $\Gt$ is an extension of $A+B$. It remains to show that $\Gt$ generates a substochastic semigroup. To do so, we check that $\Gt$ satisfies Proposition \ref{SubstGen}. Condition (i) is satisfied since for all $u \in D(G)_+$,
\[\inp{\Psi, \Gt u} = \inp{\Psi, Gu} + (a_0-\ab)(u) \inp{\Psi, u_0}= -a_0(u)(1-\norm{u_0}) - \ab(u) \norm{u_0}\leq 0\]
as $a_0,\ab$ are positive functionals and $\norm{u_0} \leq 1$. To show that condition (ii) is satisfied, substitute $\Gt$ into \eqref{E-ResSoln} to get 
\begin{equation}\label{E-UniqRes1}
(\ld -G)v+(\ab-a_0)(v)u_0 = u.
\end{equation} 
Applying $R(\ld, G)$ to both sides and rearranging, we have
\begin{equation}\label{E-UniqRes2}
v = R(\ld, G)u + \ap_u R(\ld, G)u_0  
\end{equation}
where $\ap_u$ is some constant depending on $u$.

To see that $\ap_u$ is unique for every $u$, substitute \eqref{E-UniqRes2} into \eqref{E-UniqRes1}.
%to get \[\ap_u(1+(\ab-a_0)(R(\ld, G)u_0))u_0 = (a_0-\ab)(R(\ld, G)u)u_0.\]
Then $u_0 \neq 0$ implies
\[\ap_u(1+(\ab-a_0)(R(\ld, G)u_0)) = (a_0-\ab)(R(\ld, G)u).\]
Now consider the coefficient of $\ap_u$. By the definition of $a_0$ we have 
\begin{align*}
1+(\ab-a_0)(R(\ld, G)u_0) &= 1+\ab(R(\ld, G)u_0)+\inp{\Psi, \ld R(\ld, G)u_0-u_0}\\
&=1-\norm{u_0}+\norm{\ld R(\ld, G)u_0}+\ab(R(\ld, G)u_0) 
\end{align*}
which is strictly positive since $\norm{u_0}\leq 1$, $\ab$ is positive, $u_0 \in X_+\backslash \{0\}$ and $R(\ld, G)$ is injective. Therefore $\ap_u$ exists and is unique for all $u \in X$. Moreover, $\ap_u>0$ if $u \in X_+$ because $a_0 \geq \ab$. Hence, the solution $v$ in \eqref{E-UniqRes2} to \eqref{E-ResSoln} is unique and moreover, positive if $u \in X_+$, and so condition (ii) of Proposition \ref{SubstGen} is satisfied. Therefore $\Gt$ generates a substochastic
semigroup. Since $u_0$ was an arbitrary positive element with $\norm{u_0} \leq 1$, it follows that we can construct infinitely many semigroups of this form by
varying $u_0$.
\end{proof}

The proof above shows that if $\sg{V}$ is dishonest and $X$ is not 1-dimensional, there are in fact infinitely many substochastic semigroups $\sg{\Vt}$ with
generator $\Gt \supset A+B$ whose loss is ``minimal" in the sense 
\begin{equation}\label{E-infmany}
\norm{\Vt(t)u} - \norm{u} = -\ab\left(\int_0^t \Vt(s) u \,ds \right)\quad \text{for all } u \in X_+.
\end{equation}
To see this, observe that \eqref{E-GNonUniq} can be rewritten as
\[\Gt u = Gu-\inp{\Psi, Gu}u_0-\ab(u)u_0 \quad\text{ for all }u \in D(G)\]
where $\norm{u_0} \leq 1, u_0 \in X_+\backslash \{0\}$. Taking $u_0$ satisfying $\norm{u_0}=1$, we have for all $u \in D(G)$,
\[\inp{\Psi,\Gt u} = (1-\norm{u_0})\inp{\Psi, Gu}-\ab(u)\norm{u_0} = -\ab(u)\]
and so it follows that \eqref{E-infmany} holds.

Theorem \ref{H-UAB} is a generalisation of \cite[Theorem 6]{R} where Reuter was interested in uniqueness of solutions to the backward Kolmogorov
differential equations. As Kato's Theorem originated from studying solutions to Kolmogorov differential equations, it is unsurprising that we have an analogous result
regarding the uniqueness of solutions in Kato's setting too.

%\section{Preserving Honesty}\label{S-OtherHT}
\section{Honesty and Potentials}\label{S-HonPot}
%{\color{red}--relate this to study of absorption sgs}
In the study of semigroup properties, one is often interested in the preservation of the properties under modifications to the semigroup. In this section, we look at the preservation of honesty of the perturbed Kato semigroup under another class of perturbations, namely the addition of an absorption term.

An absorption or potential term is a common term which occurs in differential equations which model dynamical systems. In semigroup language, this is often phrased in terms of adding, or more precisely, taking away a (positive) potential term from the generator of the original semigroup. An example of this can be seen in the transport equation studied in \cite[Section 3]{V87} (see also Example \ref{EgCTransp}). The transport operator $T:=T_0-h$ where $T_0$ is the free-streaming operator and $0 \leq h$ is the absorption term. In this case, the free-streaming operator generates the original semigroup $\sg{U}$ while $T_0-h$ generates the new absorption semigroup, $\sg{U_h}$. Note that the relation $U_h(t) \leq U(t)$ holds for all $t \geq 0$ as $h$ is positive. Other examples where absorption terms occur include piecewise deterministic Markov processes \cite{TK09} and the heat equation on graphs \cite{KL}. Thus in this section, we study conditions which ensure that the honesty or dishonesty of the original semigroup is retained by the absorption semigroup. 

\begin{prop}\label{HonPot} Let $X$ be an abstract state space and suppose $A,B$ satisfy the conditions of Kato's Theorem with honest perturbed semigroup $\sg{V}$. Let $K$ be a positive operator such that there is an extension $A_K$ of $(A-K, D(A) \cap D(K))$
that generates a substochastic semigroup and $A_K, B$ also satisfy Kato's Theorem with perturbed semigroup $\sg{V_K}$. If  $(D(A) \cap D(K))_+$ is dense in the graph norm in $D(A)_+$, then $\sg{V_K}$ is also honest.
\end{prop} 
\begin{proof} Let $S:= (\ld-A)(D(A) \cap D(K))_+$. Elementary calculations show that if $(D(A) \cap D(K))_+$ is dense in
$D(A)_+$, then $S$ is dense in $X_+$.

Let $v \in S$. Then 
\begin{align*}
 &(R(\ld, A) - R(\ld, A_K))v  \\
=& R(\ld, A_K)(\ld-A_K-(\ld-A))R(\ld, A)v\\
=& R(\ld, A_K)(\ld-A+K -(\ld-A))R(\ld, A)v \quad\text{(as $R(\ld, A)v \in (D(A) \cap D(K))_+$)}\\
=&R(\ld, A_K)KR(\ld, A)v \geq 0.
\end{align*}
Since $S$ is dense in $X_+$, $R(\ld, A), R(\ld, A_K)$ are bounded and $X_+$ is closed, it follows that $R(\ld, A_K) \leq R(\ld, A)$. Hence $BR(\ld, A_K) \leq
BR(\ld, A)$ and iterating, $(BR(\ld, A_K))^n \leq (BR(\ld, A))^n$. The result now follows by Theorem \ref{HT-R} (ii).
\end{proof} 

Note that if $K$ is a bounded positive operator and $A, B$ satisfy Kato's Theorem, then $A-K, B$ also satisfy Kato's Theorem since $(A-K, D(A))$ also generates a
substochastic semigroup and $\inp{\Psi, (A-K+B)u} = \inp{\Psi, (A+B)u }-\inp{K,u} \leq 0$ for all $u \in D(A)_+$. We will denote the perturbed semigroup by
$\sg{V_K}$. Proposition \ref{HonPot} tells us that honesty is retained even after adding a bounded potential term to the generator. It turns out that in this case, dishonesty is retained as well:
\begin{prop}\label{DishonPot} Let $X$ be an abstract state space. Suppose $A,B$ satisfy the conditions of Kato's Theorem with dishonest perturbed semigroup $\sg{V}$ and let $K$ be a bounded positive operator. Then the perturbed semigroup
$\sg{V_K}$ is also dishonest.
\end{prop}
\begin{proof} Theorem \ref{H-UAB} (ii) implies that there are infinitely many extensions $G_\ap$ of $A+B$ which generate substochastic semigroups. Since $K$ is
positive and bounded, $G_\ap-K$ also generates a substochastic semigroup for each $\ap$. Moreover, since $G_\ap \supseteq A+B$ and $K$ is bounded, it follows that
$G_\ap-K \supseteq A-K+B$. Hence there exist infinitely many substochastic semigroups whose generators are extensions of $A-K+B$. So by Theorem \ref{H-UAB} (ii),
$\sg{V_K}$ is also dishonest.   
\end{proof}

It turns out that for the case of preserving dishonesty, adding a bounded potential is in some sense sharp. We can find an example that shows that a dishonest semigroup can be converted into an honest semigroup by adding a potential that is ``large enough". We will not go into the full details here as it requires a lot more auxiliary information, but simply give a brief outline. This example can be found in the study of Laplacians on graphs where the Laplacian is known to generate a substochastic heat semigroup on $\ell^1$. In \cite{W-SC}, it is shown that the heat semigroup can be seen as a perturbed semigroup derived from Kato's Theorem and moreover honesty of the heat semigroup is equivalent to stochastic completeness of the graph. Then \cite[Theorem 2]{KL} states that any graph can be modified to form a stochastically complete one by adding a (sufficiently large) potential term to the Laplacian. In other words, the heat semigroup generated by any graph Laplacian can be modified to become an honest one by increasing the potential term adequately. So dishonesty is not always preserved under absorption.

Finally, let us apply these results to the case of the transport equation.
\begin{eg}\label{EgCTransp} Let us consider the linear transport equation with no incoming particles as boundary condition from \cite[Section 3]{V87}. So we let $X =L^1(S \x V, \mu)$ where $S \subset \R^n$, $V\subset \R^n$ are locally compact in the induced topology and $\mu = \ld^n \x \rho$ where $\ld^n$ is the $n$-dimensional Lebesgue measure and $\rho$ is a locally integrable Borel measure on $V$. The linear transport equation is given as 
\[\frac{\partial f}{\partial t}(s,v,t)= -v \cdot \nabla_{\bf x} f(s,v,t) -h(s,v)f(s,v,t) +\int_V k(s,v,v')f(s,v',t)\,d\rho(v'), \quad {\bf x} = (s,v)\]
where $h:S \x V \to [0,\infty]$ and $k:S\x V \x V \to [0,\infty]$ are measurable functions.
 
Let $T_0$ denote the generator of the $C_0$-semigroup of free streaming $\sg{U_0}$ and $h$ denote the maximal multiplication operator associated with the function $h(s,v)$. In this case, we will assume that $T= T_0-h$, $D(T)=D(h)$, generates the substochastic $C_0$-semigroup $\sg{U}$ and $K$ is the positive operator defined by 
\[Kf(s,v):=\int_V k(s,v,v')f(s,v')\,d\rho(v'),\quad f \in D(K)=D(T)\]
with
\begin{equation*}\label{E-PotEgh}
\norm{hf} \leq - \int Tf\, d\mu,\quad \text{for all }f \in D(T)_+
\end{equation*}
and
\begin{equation}\label{E-PotEgk}
\int_V k(s,v,v')\, d\rho(v') \leq h(s,v) \qquad \mu\text{-a.e.}
\end{equation}
Under this setup, $T$ and $K$ satisfy the conditions of Kato's Theorem with perturbed semigroup $\sg{V}$ \cite[p.463]{V87}.

Next we consider a second measurable function $\tilde{h}(s,v) \geq 0$ and suppose that $\tilde{h}$ (the maximal multiplication operator with $\tilde{h}(s,v)$) is $U(\cdot)$-bounded (see \cite[Definition 1.2]{V86}) i.e. $\tilde{h}$ is $T$-bounded and there exist $\ap \in (0, \infty]$ , $\gamma \geq 0$ such that for all $f \in D(T)$ 
\[\int_0^\ap \norm{BU(t)f} \,dt \leq \gamma \norm{f}.\]
By \cite[Corollary 2.10]{V86}, $T_{\tilde{h}} = T-\tilde{h}$ generates a substochastic $C_0$-semigroup. Moreover for $f \in D(T)_+$,
\[ \norm{(h+\tilde{h})f} \leq -\int Tf \,d\mu+ \int \tilde{h}f \, d\mu = -\int T_{\tilde{h}}f \,d\mu.\]
and \[\norm{Kf} \leq \norm{(h+\tilde{h})f}.\]

Therefore, $T_{\tilde{h}}$ and $K$ also satisfy the conditions of Kato's Theorem with perturbed semigroup denoted $\sg{V_{\tilde{h}}}$. Since by assumption $\tilde{h}$ is $T$-bounded, we have $(D(T) \cap D(\tilde{h}))_+$ dense in $D(T)_+$, so by Proposition \ref{HonPot}, $\sg{V_{\tilde{h}}}$ is honest whenever $\sg{V}$ is.

In the case when we have equality in \eqref{E-PotEgk}, the background material is called a pure scatterer \cite[p.463]{V86}. Proposition \ref{HonPot} and \ref{DishonPot} tell us that adding a bounded $\tilde{h}$ to $h$ does not affect the honesty or dishonesty of the transport semigroup. In other words, a change from pure scatterer to impure does not affect honesty if the change is ``small" enough. 
\end{eg}
In the next section, we will look at how the abstract results in Sections \ref{SS-K-ASS} to \ref{S-HonPot} can be applied to study quantum dynamical semigroups.

\section{Quantum Dynamical Semigroups and Kato's Theorem}\label{S-KQDS}
Kato's original theorem on $L^1$ is connected to the study of stochastic processes and classical Markov semigroups. The extension of Kato's Theorem to abstract state spaces allows us to apply this theory to the non-commutative setting. In this section and the next, we will demonstrate an application of Kato's Theorem to quantum dynamical semigroups. We will begin by introducing a special class of quantum dynamical semigroups in this section and describe how they can be constructed using Kato's Theorem. As noted in the introduction, although Kato's methods have been employed by Davies to construct quantum dynamical semigroups in \cite{D77}, the theory of quantum dynamical semigroups has largely developed independently of Kato's Theorem. Hence, we will begin by presenting the theory of quantum dynamical semigroups independently of Kato's Theorem following the survey of Fagnola \cite{F99} in Section \ref{SS-SC} before applying Kato's Theorem in Section \ref{S-QDSKato}. In Section \ref{S-HTQDS}, we will investigate applications of honesty theory for quantum dynamical semigroups.  

\subsection{Quantum Dynamical Semigroups}\label{SS-SC} 
The brief summary of the theory of quantum dynamical semigroups presented in this section is based on \cite{F99} where Fagnola considers quantum dynamical semigroups defined on the space of bounded operators on a complex Hilbert space $\Hb$, $\lop(\Hb)$. In particular, he constructs a minimal quantum dynamical semigroup based on Chung's construction of the minimal solution of Feller-Kolmogorov equations for countable state Markov chains. Note that in the rest of this chapter, $\Hb$ will denote a complex Hilbert space and $\inp{\cdot,\cdot}$ will denote the inner product on $\Hb$. Also, $\lop(\Hb)_+$ is the cone consisting of $\Hb$-positive operators i.e.\ operators $T$ such that $\inp{Tx,x} \geq 0$ for all $x \in D(T)$.

We begin by presenting some preliminary information. First recall that $\loph$ has predual isometrically isomorphic to $\Tf$, the space of trace class operators on $\Hb$, equipped with the trace norm, $\norm{\cdot}_{tr}$. The duality is given by 
\[\inp{\rho,x}_{\Tf,\loph} = \Tr(\rho x),\qquad \rho \in \Tf, x \in \loph.\]
This duality will allow us to apply Kato's Theorem (on abstract state spaces) and its related results to the theory of quantum dynamical semigroups on $\loph$. This follows because the predual space $\Tf$ is the complexification of the space of self-adjoint trace-class operators $\Tsa$, which is a real ordered Banach space with trace norm additive on the positive cone, i.e.\ it is an abstract state space. The functional $\Psi$ which we saw in Section \ref{SS-K-ASS} is simply the trace functional in this case. 

The subspace of $\Tf$ consisting of the rank one operators $\ruv$, $u, v \in \Hb$, defined by 
\[\ruv \varphi:= \inp{v,\varphi}u,\qquad \varphi \in \Hb\]
will play an important role. %Note that for $u, v \in \Hb$, $x \in \lop(\Hb)$, we have $\Tr(x\ruv) = \inp{v,xu}$. Moreover, $\norm{\ruv}_{tr} = \norm{u}\norm{v} = \norm{\ruv}_\infty$ where $\norm{\cdot}_\infty$ refers to the usual operator (sup-)norm.
We will require the following lemma about the convergence of rank one operators.
\begin{lem}\label{conroo} Suppose $u, v, (u_n), (v_n) \in \Hb$ satisfy $\norm{u_n - u} \to 0$ and $\norm{v_n - v} \to 0$ as $n \to \infty$. Then $\norm{\ruv - \roo{u_n}{v_n} }_{tr} \to 0$ as $n \to \infty$.  
\end{lem}

%Let us also clarify terminology we will use to describe some topologies of operator algebras. 
%\begin{defn} Let $(x_\ap)$ be a net in $\lop(\Hb)$ and let $x \in \lop(\Hb)$. We say that 
%\begin{enumerate} 
%%\item $(x_\ap)$ converges weakly to $x$ if $\inp{v,x_\ap u} \to \inp{v,xu}$ for every $v,u \in \Hb$.
%\item $(x_\ap)$ converges $\sm$-weakly to $x$ if the sum $\sum_n \inp{v_n,x_\ap u_n} \to \sum_n \inp{v_n,x u_n}$ for every pair of sequences $(v_n), (u_n)$ of
%elements of $\Hb$ such that the series $\sum_n\norm{v_n}^2$ and $\sum_n\norm{u_n}^2$ converge.
%\item $(x_\ap)$ converges strongly to $x$ if $x_\ap u \to xu$ for every $u \in \Hb$.
%\end{enumerate}
%\end{defn}
%Note that $\sm$-weak convergence is simply weak$^*$-convergence, i.e. $(x_\ap)$ converges $\sm$-weakly to $x$ if and only if for every trace class operator
%$\rho \in \mathfrak{T}(\Hb)$, $\Tr(x_\ap\rho)$ converges to $\Tr(x\rho)$. %Moreover, the weak and $\sm$-weak topology coincide on bounded subsets of $\lop(\Hb)$.
%Finally, we introduce the notion of complete positivity:
%\begin{defn}\cite[Definition 2.5]{F99} Let $\A$ and $\B$ be two $^*$-sub-algebras of $\lop(\Hb)$.
%The linear map $\Tc:\A \to \B$ is called completely positive if for every $n \in \N$, and every family $a_1, \dots, a_n$ of $\A$ and every family $b_1, \dots,
%b_n$ of $\B$, we have \[\sum_{p,q=1}^n b_p^*\Tc(a_p^*a_q)b_q \geq 0.\] 
%\end{defn}

Now we are ready to define quantum dynamical semigroups.

\begin{defn}\label{QDSdef}\cite[Definition 3.1]{F99} Let $\A$ denote a $W^*$-algebra of operators acting on a Hilbert space $\Hb$. A quantum dynamical semigroup on $\A$ is a family $\sg{\Tc}$ of bounded operators on $\A$ with the following properties:
\begin{enumerate}
\item $\Tc(0)a = a$ for all $a \in \A$.
\item $\Tc(t+s)a = \Tc(t)\Tc(s)a$ for all $s,t \geq 0$ and all $a \in \A$.
\item $\Tc(t)$ is completely positive for all $t \geq 0$.
\item $\Tc(t)$ is a $\sm$-weakly continuous operator in $\A$ for all $t \geq 0$.
\item For every $a \in \A$, the map $t \mapsto \Tc(t)a$ is continuous with respect to the $\sm$-weak topology on $\A$.
\end{enumerate}
\end{defn}

\begin{defn}\cite[Definition 3.2]{F99} The infinitesimal generator of the quantum dynamical semigroup $\sg{\Tc}$ is the operator $\G^*$ whose domain $D(\G^*)$ is the space of elements $a \in \A$ for which there exists an element $b \in \A$ such that $b = \lim_{t \to 0} \frac{\Tc(t)a-a}{t}$ in the $\sm$-weak topology and $\G^* a=b$. 
\end{defn}

Since the quantum dynamical semigroup $\sg{\Tc}$ on $\loph$ satisfies conditions (iv) and (v) of Definition \ref{QDSdef}, it follows that $\sg{\Tc}$ induces a
predual semigroup $\sg{S}$ on $\Tf$ defined by 
\[\inp{\SC(t)\rho,x}_{\Tf, \loph} = \inp{\rho, T(t)x}_{\Tf, \loph} \qquad\text{for all }\rho \in \Tf, x \in \loph, t\geq 0.\]
Equivalently, this may be stated via the generator of the semigroup, i.e.\ $\G$ is the generator of $\sg{\SC}$ if and only if $\G^*$ is the generator of $\sg{\Tc}$
\cite[Theorem 1.2.3]{vN}, \cite[p.252]{BR}. If $\sg{\Tc}$ is a quantum dynamical semigroup, the predual semigroup is in fact, strongly continuous. This follows from condition (v) in Definition \ref{QDSdef}. To ensure that the notation in this section concurs with those in the previous sections, we will always denote the generator of a quantum dynamical semigroup as an adjoint operator, for example $\G^*$; more precisely, as the adjoint of the generator of the predual semigroup.

The special class of quantum dynamical semigroups we are interested in satisfies the following premise, which we will assume holds for the remainder of this paper unless stated otherwise:
\begin{hyp}\label{QDShyp} Suppose $Y$ generates a $C_0$-semigroup of contractions $\sg{P}$ in $\Hb$. Suppose also the sequence of operators $(L_l)_{l=1}^\infty$
are such that $D(L_l) \supseteq D(Y)$ and for all $u \in D(Y)$, we have 
\begin{equation}\label{E-QDS}
\inp{u,Yu}+\inp{Yu,u}+ \sum_{l=1}^\infty \inp{L_l u,L_lu} \leq 0.
\end{equation}
\end{hyp}
It will also prove useful later to consider the sesquilinear form $\Ups(x)$, $x \in \loph$ with domain $D(Y) \x D(Y) \subseteq \Hb \x \Hb$ given by 
\begin{equation}\label{E-QDSGen}
\Ups(x)[v, u] = \inp{v,xYu}+\inp{Yv,xu}+ \sum_{l=1}^\infty \inp{L_lv,xL_l u}.
\end{equation}
%Under the right conditions, the form $\Ups(x)$ will be induced by a closed linear operator. In these cases, we will denote the closed operator associated with it by $W(x)$. 
Assuming Premise \ref{QDShyp} holds, Fagnola \cite[Chapter 3]{F99}, shows that: 
\begin{prop}\label{QDSGen} \cite[Theorem 3.22]{F99}
There exists a minimal quantum dynamical semigroup $\sg{T}$ satisfying 
\begin{equation}\label{E-QDSGenT}
\inp{v,(T(t)x) u} = \inp{v,xu}+\int_0^t \Ups(T(s)x)[v,u]\,ds \quad\text{for all }u,v \in D(Y)
\end{equation}
 and $T(t)\One \leq \One$  for all $t \geq 0$. The semigroup is minimal in the sense that for any quantum dynamical semigroup $\sg{U}$ on $\loph$ which is a solution to \eqref{E-QDSGenT} and for any $x \in \loph_+$, we have $T(t)x \leq U(t)x$ for all $t \geq 0$.
\end{prop}

%Fagnola constructs the semigroup by defining for every $t \geq 0$, a sequence of linear contractions $\{T^{(n)}(t)\}_n$ iteratively, namely for all $t \geq 0$, $x \in \loph$ and
%$u,v \in D(Y)$,
%\begin{align*}
%\inp{v, T^{(0)}(t)x u} &= \inp{ P(t)v, xP(t)u}\\ 
%\inp{v, T^{(n+1)}(t)x u} &= \inp{ P(t)v, xP(t)u} + \sum_{l=1}^\infty \int_0^t \inp{L_l P(t-s)v, T^{(n)}(s)xL_l P(t-s)u} \,ds.
%\end{align*}
%For all $x \in \loph_+$ and every $t \geq 0$, the sequence $\{T^{(n)}(t)x\}_n$ is increasing and bounded from above by $\norm{x}\cdot \One$. Therefore,
%the limit $\lim_{n \to \infty }\inp{u, T^{(n)}(t)x u}$ exists for all $u \in \Hb$. The semigroup is then given by this limit, i.e. 
%\[\inp{v, T(t)x u}  = \lim_{n \to \infty }\inp{v, T^{(n)}(t)x u}\qquad
%\text{ for all } u, v \in \Hb.\] 

Equation \eqref{E-QDSGenT} can be restated in terms of the generator of the quantum dynamical semigroup:
\begin{prop}\label{GenForm} A contractive quantum dynamical semigroup $\sg{\Tc}$ satisfies \eqref{E-QDSGenT} if and only if its generator $\G^*$ satisfies 
\begin{equation}\label{E-LinbGen}
\inp{v, (\G^* x) u} = \Ups(x)[v,u]\quad\text{for all }u,v \in D(Y),x \in D(\G^*).
\end{equation}
\end{prop}

\begin{proof} Suppose $\sg{\Tc}$ satisfies \eqref{E-QDSGenT}. Fix $u,v \in D(Y)$ and $x \in D(\G^*)$. Since $\Tr(x\ruv) = \inp{v,xu}$, \eqref{E-QDSGenT} can be rewritten as $\Tr((\Tc(t)x)\ruv)-\Tr(x\ruv)=\int_0^t \Ups(\Tc(s)x)[v,u]\,ds$. Hence, 
\begin{align}\label{E-GenForm}\frac{1}{t}\Tr&((\Tc(t)x-x)\ruv) \notag\\
&= \frac{1}{t}\int_0^t \left(\inp{Yv,(\Tc(s)x)u}+ \inp{v,(\Tc(s)x)Yu}+ \sum_{l=1}^\infty \inp{L_lv,(\Tc(s)x)L_l u}\right)\,ds.
\end{align}
The continuity of $t \mapsto \Tc(t)x$ with respect to the $\sm$-weak topology implies that the maps $s \mapsto \inp{Yv,(\Tc(s)x)u}$, $s \mapsto \inp{v,(\Tc(s)x)Yu}$, $s \mapsto \inp{L_lv,(\Tc(s)x)L_l u}$, $l \in \N$
are continuous. Moreover, since $\sg{\Tc}$ is contractive, we have for all $l \in \N$, $\abs{\inp{L_lv,(\Tc(s)x)L_l u}} \leq \norm{x}\norm{L_lu}\norm{L_lv}$. But
by \eqref{E-QDS}, we have 
\begin{equation}\label{E-LL}
\sum_{l=1}^\infty \norm{L_lv}\norm{L_l u} \leq \left(\sum_{l=1}^\infty \norm{L_lv}^2\right)^{\frac{1}{2}}\left(\sum_{l=1}^\infty \norm{L_l u}^2\right)^{\frac{1}{2}}\leq (-2\Real\inp{v,Yv})^{\frac{1}{2}}(-2\Real\inp{u,Yu})^{\frac{1}{2}}.
\end{equation}
Thus, by the Weierstrass M-test, the map $s\mapsto \sum_{l=1}^\infty\inp{L_lv,(\Tc(s)x)L_l u}$ is continuous. Therefore we can let $t \to 0$ in \eqref{E-GenForm} to obtain $\inp{v, (\G^* x) u} = \Ups(x)[v,u]$.

Conversely, suppose $\sg{\Tc}$ satisfies \eqref{E-LinbGen}. We begin by observing that the form $\Ups(x)[v,u]$, $x \in \loph, u, v \in D(Y)$ can be restated as 
\begin{align}\label{E-UpsTrace}
\Ups(x)[v,u] &= \Tr\left(x\left(\roo{Yu}{v}+\roo{u}{Yv}+\sum_{l=1}^\infty \roo{L_lu}{L_lv}\right)\right)
%&= \inp{x, \roo{Yu}{v}+\roo{u}{Yv}+\sum_{l=1}^\infty \roo{L_lu}{L_lv} }_{\loph,\Tf}
\end{align}
since $\sum_{l=1}^\infty \roo{L_lu}{L_lv}$ converges in trace norm by \eqref{E-LL}. On the other hand, from \cite[Proposition 1.2.2]{vN}, we have for $x \in \loph$,
\[\Tc(t)x-x = \G^* \text{weak}^*\int_0^t \Tc(s)x\,ds\]
where $\text{weak}^*\int_0^t \Tc(s)x\,ds$ denotes the weak$^*$ integral of $\Tc(s)x$.
Hence 
\begin{align*} 
&\inp{v, (\Tc(t)x)u}-\inp{v,xu}\\
=& \inp{v,\left(\G^* \text{weak}^*\int_0^t \Tc(s)x\,ds\right)u}\\
=&\Ups\left( \text{weak}^*\int_0^t \Tc(s)x\,ds \right)[v,u]\qquad\text{(by assumption)}\\
=&\Tr\left(\left(\text{weak}^*\int_0^t \Tc(s)x\,ds\right)\left(\roo{Yu}{v}+\roo{u}{Yv}+\sum_{l=1}^\infty \roo{L_lu}{L_lv}\right)\right) \quad\text{ (by \eqref{E-UpsTrace})}\\
=&\int_0^t\Tr\left( (\Tc(s)x)\left(\roo{Yu}{v}+\roo{u}{Yv}+\sum_{l=1}^\infty \roo{L_lu}{L_lv}\right)\right)\,ds \quad\text{ (by definition)}\\
%&\hphantom{{}\int_0^t\Tr( (\Tc(s)x)(\roo{Yu}{v}+\roo{u}{Yv} \sum_{l=1}^\infty \roo{L_lu}{L_lv} }(\text{by definition of weak$^*$ integral})\\
=&\int_0^t \Ups(\Tc(s)x)[v,u]\,ds \qquad\text{ (by \eqref{E-UpsTrace})}. 
\end{align*}
Therefore $\sg{\Tc}$ satisfies \eqref{E-QDSGenT}.
\end{proof}

\begin{defn}\label{DefLind}
We say that the generator $\G^*$ of a quantum dynamical semigroup can be represented in Lindblad form if there exists operators $Y,$ $(L_l)$ on $\Hb$ satisfying Premise \ref{QDShyp} such that \[\inp{v,(\G^*x)u}= \Ups(x)[v,u]\]
for all $x \in D(\G^*)$ and all $u,v \in D(Y)$.
\end{defn}

\subsection{Constructing Quantum Dynamical Semigroups via Kato's Theorem}\label{S-QDSKato}
Now we look at the construction via Kato's Theorem, of a quantum dynamical semigroup whose generator can be represented in Lindblad form. In particular, we will show that the minimal quantum dynamical semigroup identified in Proposition \ref{QDSGen} coincides with that constructed via Kato's Theorem. Although this application of Kato's methods to quantum dynamical semigroups were also noted by others including Davies \cite{D77} and Arlotti, Lods and Mokhtar-Kharroubi \cite{ALM}, we have yet to find any literature which applies Kato's Theorem directly to quantum dynamical semigroups or which actually shows that the two methods are equivalent. We will prove this equivalence in this section both for completeness and as preparation for the next section on applications of honesty theory. Since Kato's Theorem is stated for real spaces, we will restrict to the space of self-adjoint trace class operators, $\Tsa$.

First, consider the semigroup $\sg{U}$ in $\Tsa$ defined by $U(t)\rho = P(t)\rho P(t)^*$. Note that since $\sg{P}$ is contractive, so is $(P(t)^*)_{t \geq 0}$. It turns out that $\sg{U}$ is also a $C_0$-semigroup of contractions with generator we will denote by $A$ (see for example \cite[Section I.3.16]{EN}). In \cite{D77}, Davies considers the case where we have equality in equation \eqref{E-QDS} and shows that the operator $A$ and an appropriately defined $B$ (see Lemma \ref{qdsB}, Corollary \ref{Bext}) satisfy in our terminology, the conditions in Theorem \ref{K-Ass}. His methods also hold for the more general case (with inequality in \eqref{E-QDS}) with minor modifications. We describe his methods below as this setup will be used in Section \ref{S-HTQDS} as well.

To determine the domain of the generator $A$, Davies introduces a positive, one-to-one map 
\[\pi:\Tsa \to \Tsa, \quad \pi(\rho) = R(1,Y)\rho R(1, Y)^*\]
and considers the subspace $\Dc_s:=\pi(\Tsa)$. Then \cite[Lemma 2.1]{D77} (restated as Lemma \ref{qdsA}) holds in this case as well since the inequality \eqref{E-QDS} has no role in the proof.
\begin{lem}\label{qdsA} \cite[Lemma 2.1]{D77}
The domain $\Dc_s$ is dense in $\Tsa$. Let $\rho \in \Dc_s$ and $\ep >0$. Then there exist $\rho_1,\rho_2 \in (\Dc_s)_+:=\Dc_s \cap \Tsa_+$ such that 
\begin{equation}\label{E-rho}
\rho = \rho_1-\rho_2, \quad \norm{\rho_1}_{tr}+\norm{\rho_2}_{tr}<\norm{\rho}_{tr} +\ep. 
\end{equation}
Moreover, $\Dc_s$ is a core for $A$ and for all $\rho \in \Dc_s$
\begin{equation}\label{E-qdsA}
A\rho= Y\rho+\rho Y^*
\end{equation}
in the sense that $Y\rho$ is a trace class operator while $\rho Y^*$ is a restriction of the operator $(Y\rho)^*$ which is also trace class.  
\end{lem}

Now let us consider the operator $B$. The next two lemmas are the analogues of \cite[Lemma 2.2, Lemma 2.3]{D77} and can be proven almost exactly as in \cite{D77}. The only changes required are changes from equalities to inequalities at the appropriate points, hence the proofs are omitted.
\begin{lem}\label{qdsB}
 The formula 
\begin{equation}\label{E-qdsB}
B\rho=\sum_{l=1}^\infty L_lR(1,Y)\pi^{-1}(\rho)(L_lR(1,Y))^*
\end{equation}
with the series converging in the trace norm defines a positive linear map $B:\Dc_s \to \Tsa$ such that 
\begin{equation}\label{E-ABineq}
\Tr(A\rho+B\rho) \leq 0\qquad\text{for all }\rho \in (\Dc_s)_+.
\end{equation}
\end{lem}
\begin{lem}\label{J} For all $\ld>0$, the map $BR(\ld, A)$ from $\Dc_s$ into $\Tsa$ has a unique, positive, bounded linear extension $J_\ld: \Tsa \to \Tsa$ such that $\norm{J_\ld} \leq 1$.
\end{lem}
\begin{rem}\label{BAbdd} Since $A$ is resolvent positive, $R(\ld,A)\Dc_s\subset \Dc_s$ and $B$ is positive on $\Dc_s$, it follows from \eqref{E-ABineq} that for all $\rho \in (\Dc_s)_+$, \[\norm{\BR \rho}_{tr} \leq -\Tr(AR(\ld, A)\rho) = \Tr(\rho)-\ld \Tr(R(\ld, A)\rho) \leq \norm{\rho}_{tr}.\] Then \eqref{E-rho} implies that $\norm{\BR \rho}_{tr} \leq \norm{\rho}_{tr}$ for all $\rho \in \Dc_s$, that is, $B$ is $A$-bounded on $\Dc_s$. Davies uses this to prove the existence of $J_\ld$ in Lemma \ref{J}.
\end{rem}

The results above allow us to derive an important corollary. We give the complete proof here as some details were omitted in \cite{D77}.
\begin{cor}\label{Bext} The map $B$ has a positive extension $B':D(A) \to \Tsa$ such that 
\begin{equation}\label{E-ABineqX}
\Tr(A\rho+B' \rho)\leq 0\qquad\text{for all }\rho \in D(A)_+.
\end{equation}
\end{cor}
\begin{proof} We define \[B'\rho = J_1(I-A)\rho.\] We begin by showing that $B'$ is an extension of $B$. Since $B$ is $A$-bounded on $\Dc_s$ (by Remark \ref{BAbdd}), it suffices to show that $B'\rho =B\rho$ for all $\rho$ in a core of $A$ which lies in $\Dc_s$. In particular, we will show that $B'\rho =B\rho$ for all $\rho \in \pi^2(\Tsa) \subset \Dc_s$. To see that $\pi^2(\Tsa)$ is a core for $A$, simply note that density in $\Tsa$ follows because $\pi^2(\Tsa)$ contains the finite rank operators whose eigenvectors lie in $D(Y^2)$ (see also Lemma \ref{r1coreA}) while invariance of $\pi^2(\Tsa)$ under $\sg{U}$ follows directly from the definition of $\pi$. So let $\rho = \pi^2(\sm)$ for some $\sm \in \Tsa$, that is, $\rho \in \pi^2(\Tsa)$. Then by \eqref{E-qdsA}, we have that $A \rho = \pi(Y\pi(\sm)+\pi(\sm)Y^*)\in \Dc_s$ and so $(I-A)\rho \in \Dc_s$. Therefore $B'\rho = J_1(I-A)\rho = BR(1,A)(I-A) \rho = B\rho$ and so $B'$ is an extension of $B$. Moreover, $\norm{B'R(1, A)\rho}_{tr}= \norm{J_1 \rho}_{tr} \leq \norm{\rho}_{tr}$ for all $\rho \in \Tsa$. Therefore $B'$ is $A$-bounded. 

To show that the inequality \eqref{E-ABineqX} holds, let $\rho \in D(A)_+$ and consider 
\[\rho_\ep = R(1, \ep Y)\rho R(1, \ep Y)^*, \qquad \ep>0.\]
It is easy to see that $\rho_\ep$ is self-adjoint if $\rho$ is. Moreover, 
$(1-\ep^{-1})R(1,Y)R(\ep^{-1},Y)+R(1,Y)=R(\ep^{-1},Y)$, so it follows that
\[\rho_\ep = R(1,Y)((1-\ep^{-1}) R(1,\ep Y)+\ep^{-1}I)\rho ((1-\ep^{-1}) R(1,\ep Y)^*+ \ep^{-1}I)R(1,Y)^*.\]
%\rho_\ep &= \ep^{-2}R(1,Y)(I-Y) R(\ep^{-1},Y)\rho R(\ep^{-1},Y)^*(I-Y)^*R(1,Y)^*\\
Thus, $\rho_\ep \in (\Dc_s)_+$.
Moreover, the map $\rho \mapsto \rho_\ep$ is bounded independently of $\ep$ as 
\[\norm{\rho_\ep}_{tr} \leq \norm{R(1,\ep Y)}_\infty \norm{\rho}_{tr} \norm{R(1,\ep Y)^*}_\infty \leq \norm{\rho}_{tr}. \]
We will show that $\rho_\ep \to \rho$ as $\ep \to 0$ in trace norm for all $\rho \in \Tsa$. Consider the rank one operator $\rho:=\ruv, u, v \in \Hb$. Elementary calculations show that $\rho_\ep = \roo{u_\ep}{v_\ep}$ where $u_\ep = R(1,\ep Y)u$ and $v_\ep = R(1, \ep Y)v$. By \cite[Lemma II.3.4]{EN}, $u_\ep \to u$ and $v_\ep \to v$ as $\ep \to 0$. Hence by Lemma \ref{conroo}, it follows that $\rho_\ep \to \rho$ in trace norm as $\ep \to 0$. Since the (self-adjoint) finite rank operators are dense in $\Tsa$ and the map $\rho \mapsto \rho_\ep$ is uniformly bounded, it follows that $\rho_\ep \to \rho$ for all $\rho\in \Tsa$. 

Now for $\rho \in D(A)$, we have $A\rho \in \Tsa$ and $A\rho_\ep=(A\rho)_\ep$. Therefore we can conclude that $\rho_\ep \to \rho$ and $A\rho_\ep \to A\rho$ as $\ep \to 0$. Since $B'$ is $A$-bounded, it follows that $B\rho_\ep \to B'\rho$. Therefore by Lemma \ref{qdsB}, for all $\rho \in D(A)_+$, $B'\rho \geq 0$ and
\[\Tr(A\rho+B'\rho) = \lim_{\ep \to 0}\Tr(A\rho_\ep+B\rho_\ep) \leq 0.\]
\end{proof}
Henceforth we will identify $B$ with $B'$ and simply denote it by $B$. With this we have: 
\begin{prop}\label{QDSKato} $A$ and $B$ satisfy the conditions of Theorem \ref{K-Ass} and so there exists a minimal perturbed semigroup $\sg{\St}$ with generator $\Gt$ an extension of $A+B$.
Moreover, $R(\ld, \Gt)$ satisfies \[R(\ld,\Gt)\rho = \sum_{k=0}^\infty R(\ld, A)(\BR)^k\rho \text{ for all }\rho \in \Tsa.\] 
\end{prop}

We have just described two methods of constructing a minimal quantum dynamical semigroup with $Y, (L_l)$ satisfying Premise \ref{QDShyp}; one via Fagnola's method (Proposition \ref{QDSGen}) and the other via Kato's Theorem (Proposition \ref{QDSKato}). The remainder of this section will be devoted to showing that the two semigroups coincide. 

Note that the semigroup from Kato's Theorem acts in the space $\Tsa$ while Fagnola's semigroup acts in the space $\loph$. In order to show that the semigroups coincide, we will first transfer the semigroups to the same space, namely $\Tf$. Recalling that $\Tf$ is simply the complexification of $\Tsa$, we will henceforth work with the complexifications of the operators $A, B, \sg{U}, \sg{\St}, \Gt$ but retain the same notation. We can do so because we saw in Section \ref{SS-K-ASS} that honesty in the complexified space is equivalent to honesty in the real space. To transfer Fagnola's semigroup to $\Tf$ on the other hand, we will utilise the fact that every quantum dynamical semigroup on $\loph$ induces a predual semigroup on $\Tf$. We will denote by $\sg{S}$ the predual semigroup of Fagnola's minimal quantum dynamical semigroup $\sg{T}$ identified in Proposition \ref{QDSGen}. We now show that $\sg{S}$ coincides with $\sg{\tilde{S}}$.
\begin{thm}\label{QDS-K} Let $\sg{S}$ be the predual semigroup of the minimal quantum dynamical semigroup $\sg{T}$ in Proposition \ref{QDSGen} and $\sg{\tilde{S}}$ be the perturbed semigroup in Proposition \ref{QDSKato}. Then $\St(t)\rho = S(t)\rho$ for all $\rho \in \Tf$, $t \geq 0$. 
\end{thm}
In order to prove the theorem, we require some auxiliary information. First let us consider a few important subspaces, beginning with
\[\V=\V_1:=\Sp\{\ruv\,:\, u,v \in D(Y)\}.\] We will also occasionally require the spaces 
\[\V_n:=\Sp\{\ruv\,:\, u,v \in D(Y^n)\}, n \in \N, n \geq 2.\]
Moreover, the map $\pi$ can be extended to $\Tf$ and we will be interested in the spaces 
\[\Dc := \pi(\Tf)= \Dc_s+i\Dc_s,\quad\pi^n(\Tf), n \in \N, n \geq 2\]
\begin{lem}\label{r1coreA} For all $n \in \N$, $\V_n\subset \Dc$ and moreover, $\V_n$ is a core for $A$.
\end{lem}
\begin{proof} Since $\V_{n+1} \subseteq \V_{n}$, it suffices to show that $\V \subseteq \Dc$ to prove the first statement. Fix $u, v \in D(Y)$. Then  $\ruv = R(1, Y)\rho R(1, Y)^*$ where $\rho$ is the rank-one operator defined by $\rho :=\roo{(I-Y)u}{(I-Y)v}$. Therefore $\ruv \in \Dc$ and so $\V \subset \Dc$.

Next, we show that $\V$ is dense in $\Tf$. Since $D(Y)$ is dense in $\Hb$, it follows from Lemma \ref{conroo} that $\V$ is dense in the space of finite rank operators. Since the finite rank operators are dense in $\Tf$, it follows that $\V$ is dense in $\Tf$.

Finally, observe that $U(t)(\ruv) = P(t)\ruv P(t)^* = \roo{P(t)u}{ P(t)v}$ for all $t \geq 0$. Since $P(t)u \in D(Y)$ for all $u \in D(Y)$ and all $t \geq 0$, we have that $U(t)(\ruv) \in \V$. Therefore $\V$ is invariant under $\sg{U}$ and so $\V$ is a core for $A$.

A similar argument shows that $\V_n$ is a core for $A$ for all $n \geq 2$ since $D(Y^n)$ is a core for $Y$ and $D(Y^n)$ is invariant under the semigroup $\sg{P}$.  
\end{proof}
\begin{rem}\label{PiCore} Recall that in the proof of Corollary \ref{Bext} we showed that $\pi^2(\Tsa)$ is a core for $A|_{\Tsa}$. We can in fact show more generally that $\pi^n(\Tf)$, $n \in \N$ are cores for $A$. A similar proof as that of Lemma \ref{r1coreA} shows in fact that $\V_n \subseteq \pi^n(\Tf)$ for all $n \geq 2$. Furthermore, it is easy to see that $\pi^n(U(t)\sm) = U(t)(\pi^n\sm)$ for all $\sm \in \Tf$ and all $t \geq 0$. Therefore $\pi^n(\Tf)$, $n \in \N$ are also cores for $A$.  
\end{rem}

It will also be useful to know how the operators $A, B$ act on the operators in $\V$.
\begin{lem}\label{ABr1} For all $\ruv \in \V$ and $x \in \loph$
\[\Tr(x(A\ruv))=\inp{Yv,xu}+ \inp{v,xYu}\quad\text{and}\quad \Tr(x (B\ruv)) =\sum_{l=1}^\infty \inp{L_lv,xL_l u}.\] In particular, \[\Tr(x\left((A+B)\ruv\right)) = \Ups(x)[v,u].\] 
\end{lem}
\begin{proof} 
Fix $\ruv \in \V$ and $x \in \loph$. Note first that elementary calculations show that $Y\ruv +\ruv Y^* = \roo{Yu}{v}+\roo{u}{Yv}$. Since $\V \subset \Dc$ by Lemma \ref{r1coreA}, and hence \eqref{E-qdsA} holds for $\rho=\ruv$ (Lemma \ref{qdsA}), we have
\[\Tr(x(A\ruv))=\Tr(x(\roo{Yu}{v}+\roo{u}{Yv}))=\inp{v, xYu}+\inp{Yv,xu}.\]
On the other hand, by Lemma \ref{qdsB}, for $\rho \in \mathcal{D}$, \eqref{E-qdsB} holds and from the proof of Lemma \ref{r1coreA}, $\pi^{-1}(\ruv) = \roo{(I-Y)u)}{(I-Y)v}$. Hence for $\varphi \in
\Hb$,
\begin{align*}
(B \ruv) \varphi &= \sum_{l=1}^\infty L_lR(1,Y)\roo{(I-Y)u)}{(I-Y)v}(L_lR(1,Y))^*\varphi\\
&= \sum_{l=1}^\infty L_lR(1,Y)\inp{(I-Y)v,(L_lR(1,Y))^*\varphi}(I-Y)u\\
&= \sum_{l=1}^\infty \inp{L_l v,\varphi}L_lu = \sum_{l=1}^\infty \roo{L_lu}{L_lv}\varphi.
\end{align*}
Therefore, $\Tr(x(B\ruv))= \sum_{l=1}^\infty \inp{L_lv,xL_l u}$. The final assertion follows directly from the definition of $\Ups(x)[v,u]$. 
\end{proof}
\begin{cor}\label{KatoLind} The generator $\tilde{G}^*$ of the adjoint semigroup $\sg{\tilde{S}^*}$ of $\sg{\tilde{S}}$ can be represented in Lindblad form. 
\end{cor}
\begin{proof} Let $u,v \in D(Y)$ and $x \in D(\tilde{G}^*)$. Then by Lemma \ref{r1coreA}, $\ruv \in \V \subseteq D(A)$, so by Lemma \ref{ABr1}, it follows that $\inp{v,(\tilde{G}^*x)u} = \Tr((\tilde{G}^*x) \ruv)=\Tr(x(\tilde{G}\ruv))=\Tr(x((A+B)\ruv))=\Ups(x)[v,u]$.
\end{proof}

We also require some information about Fagnola's construction, $\sg{T}$ with generator $G^*$. In particular, Fagnola shows in \cite[Proposition 3.25]{F99} that the resolvent of $G^*$ is given by 
\begin{equation}\label{ResPQ}
R(\ld, G^*)x = \sum_{k=0}^\infty Q_\ld^k(P_\ld(x)), \quad x \in \loph, \ld>0
\end{equation} 
with the series convergent in the strong operator topology, where $P_\ld$ and $Q_\ld$, $\ld>0$ are linear positive maps in $\loph$ defined by 
\begin{align}
\inp{v, P_\ld(x)u} &= \int_0^\infty e^{-\ld s} \inp{P(s)v,xP(s)u}\,ds\label{E-Pld}\\
\inp{v, Q_\ld(x)u} &= \sum_{l=1}^\infty \int_0^\infty e^{-\ld s} \inp{L_lP(s)v,xL_lP(s)u}\,ds \label{E-Qld}
\end{align}
for $x \in \loph$, $u, v \in D(Y)$. We will rephrase $P_\ld$ and $Q_\ld$ in terms of $A$ and $B$.
\begin{lem}\label{PQBR}
 Suppose $P_\ld, Q_\ld$ are as defined in \eqref{E-Pld} and \eqref{E-Qld} and $A, B$ are as in Proposition \ref{QDSKato}. Then $P_\ld=
R(\ld,A)^* = R(\ld, A^*)$ and $Q_\ld = (BR(\ld, A))^*$ for all $\ld>0$.
\end{lem}
\begin{proof}
 Fix $\ld>0$. Since $R(\ld, A)\rho = \int_0^\infty e^{-\ld s}P(s)\rho P(s)^*\, ds$
for all $\rho \in \Tf$, it follows by elementary calculations that for $\ruv \in \V$, $R(\ld, A)\roo{u}{v} =\int_0^\infty e^{-\ld s}\roo{P(s)u}{P(s)v}\, ds$ 
where the integral is absolutely convergent in $\Tf$ and also in the graph norm of $A$.
Hence, for $u,v \in D(Y), x \in \loph$,
\begin{align*}
\inp{v, (R(\ld, A)^*x)u} &= \Tr(x(R(\ld,A) \ruv))\\
&= \Tr\left(x \int_0^\infty e^{-\ld s}\roo{P(s)u}{P(s)v}\, ds\right)\\
%&= \int_0^\infty e^{-\ld s}\Tr(x \roo{P(s)u}{P(s)v} )\,ds\\
&= \int_0^\infty e^{-\ld s}\inp{P(s)v, xP(s)u} ds =\inp{v, P_\ld(x) u}.
\end{align*}
Since $P_\ld(x) \in \loph$ for all $x \in \loph$, it follows that $P_\ld =R(\ld, A)^*$. Similarly,
\begin{align*}
\inp{v, ((\BR)^*x)u}=&\Tr(x (BR(\ld,A) \ruv))\\
=& \Tr\left(x \left(B\int_0^\infty e^{-\ld s}\roo{P(s)u}{P(s)v}\, ds\right)\right)\\
=& \int_0^\infty e^{-\ld s}\Tr(x (B\roo{P(s)u}{P(s)v}))\,ds\quad\text{(as $\roo{P(s)u}{P(s)v} \in D(B)$)}\\
=&\sum_{l=1}^\infty \int_0^\infty e^{-\ld s} \inp{L_lP(s)v,xL_lP(s)u}\,ds=\inp{v,Q_\ld(x)u }.
\end{align*}
Therefore, $Q_\ld = (BR(\ld, A))^*$.
\end{proof}

Now we can show that the two semigroups are equal.
\begin{proof}[Proof of Theorem \ref{QDS-K}] 
From \eqref{ResPQ} and Lemma \ref{PQBR}, we have that \[R(\ld, G)^*x =  R(\ld, G^*)x
= \sum_{k=0}^\infty Q^k_\ld( P_\ld(x)) = \sum_{k=0}^\infty (\BR)^{*k}R(\ld, A)^*x\] for all $x \in \loph$ with the series convergent in the strong operator
topology. Since by Proposition \ref{QDSKato}, we know that $\sum_{k=0}^\infty R(\ld, A)(\BR)^{k}\rho$ converges in trace norm for all $\rho \in \Tf$ and the trace functional is continuous on $\Tf$, it follows that \[R(\ld, G)\rho = \sum_{k=0}^\infty R(\ld, A)(\BR)^{k}\rho = R(\ld, \Gt)\rho, \quad\rho \in \Tf.\] Hence, by the Post-Widder Inversion Formula, $S(t)\rho = \St(t)\rho$ for all $\rho \in \Tf, t \geq 0$.% i.e. the predual semigroup of $\sg{T}$ satisfies Kato's perturbation theorem.
\end{proof}
In the remainder of this paper, unless stated otherwise, $\sg{T}$ will denote the minimal quantum dynamical semigroup with generator $G^*$ identified in Proposition \ref{QDSGen} with associated form $\Ups$ satisfying Premise \ref{QDShyp} and $\sg{S}$ will always denote its predual semigroup with generator $G$.
\begin{rem} One can also prove Theorem \ref{QDS-K} by using the minimality of the semigroups, that is, Fagnola's semigroup $\sg{T}$ is the minimal semigroup whose generator can be represented in Lindblad form (Proposition \ref{QDSGen}) while the Kato semigroup $\sg{\tilde{S}}$ is the minimal semigroup whose generator is an extension of $A+B$ (Theorem \ref{K-Ass}). Corollary \ref{KatoLind} tells us that the adjoint semigroup of $\sg{\tilde{S}}$ can be represented in Lindblad form. To complete the proof, we only need to show that any quantum dynamical semigroup satisfying \eqref{E-QDSGenT} has predual semigroup whose generator is an extension of $A+B$. We will in fact, prove this in the next section (Lemma \ref{QDSGAB}).       
\end{rem}

%This duality between the quantum dynamical semigroups satisfying Proposition \ref{QDSGen} and the semigroups satisfying Kato's Theorem allows us to derive some of the properties of the quantum dynamical semigroups from Kato's Theorem, for example, the minimality of the semigroup or the series representation of the resolvent of the generator. %In the next section, we will see further examples of how results related to Kato's Theorem prove useful in the study of quantum dynamical semigroups.

\section{Applications of Honesty Theory in Quantum Dynamical Semigroups}\label{S-HTQDS}
In the previous section, we saw that the generator of a certain quantum dynamical semigroup in Lindblad form can be viewed as the adjoint of a perturbed generator of a substochastic semigroup in the setting of Kato's Theorem. A natural question to investigate next is the application of honesty theory to these semigroups. It turns out that if equality holds in \eqref{E-QDS}, then honesty of the predual semigroup is equivalent to a notion known as conservativity of the quantum dynamical semigroup (Proposition \ref{HonCons}). 

Conservativity of the quantum dynamical semigroup has long been studied (see \cite{C99, FC, F99} for example). The main reason for the interest in conservativity is that it is related to the non-explosion of the system \cite{FC}, \cite[Section 3.6]{F99}. However, the study of conservativity is also of interest because conservative quantum dynamical semigroups turn out to be the semigroups with ``nice" properties. For example, if the minimal semigroup is conservative, then it is the unique semigroup satisfying \eqref{E-QDSGenT} \cite[Corollary 3.23]{F99}. Moreover, one can give a precise description of the domain of the generator if the minimal semigroup is conservative \cite[Proposition 3.33]{F99}, \cite[Theorem 4.1]{GQ}. This is important as we saw in the previous section that the domain of the generator of the quantum dynamical semigroup is difficult to determine precisely and is often defined in terms of a form.
 
We begin by giving the definition of conservativity.
\begin{defn} A quantum dynamical semigroup $\sg{T}$ is called conservative if $T(t)\One = \One$ for every $t \geq 0$.
\end{defn}
A necessary condition for $\sg{T}$ to be conservative is for 
\[\Ups(\One)[v,u] = \frac{d}{dt}\inp{v, (T(t)\One) u}|_{t=0} = 0\]  
for all $u, v \in D(Y)$. So when we speak of conservativity, we will only consider the case when we have equality in \eqref{E-QDS}. 
\begin{prop}\label{HonCons}
Suppose $\sg{T}$ is the minimal quantum dynamical semigroup identified in Proposition \ref{QDSGen} with $\Ups$ satisfying Premise \ref{QDShyp} with equality in \eqref{E-QDS}. Then $\sg{T}$ is conservative if and only if its predual semigroup $\sg{S}$ is honest.
\end{prop}
\begin{proof}
Since $\Ups$ satisfies \eqref{E-QDS} with equality, the predual semigroup being honest is equivalent to it being stochastic (Remark \ref{HonStoc}). In this context, this means that $S(t)$ is trace-preserving for all $t\geq 0$, i.e. $\Tr(S(t)\rho) = \Tr(\rho)$ for all $\rho \in \Tsa$ and this is equivalent to $\sg{T}$ being conservative since
$\Tr(S(t)\rho) = \Tr((S(t)\rho)\One)=\Tr(\rho( T(t)\One))$ for all $\rho \in \Tf, t\geq 0$.
\end{proof}
The equivalence between honesty and conservativity allows us to derive some conditions for conservativity from honesty theory by combining Theorems \ref{HT-R}, \ref{HT-ME}, \ref{HT-dual}, Proposition \ref{H-dense}, Lemma \ref{PQBR} and Proposition \ref{HonCons}. 
\begin{prop}\label{QDScons}
Suppose $\sg{T}$ is the minimal quantum dynamical semigroup identified in Proposition \ref{QDSGen} with $\Ups$ satisfying Premise \ref{QDShyp} with equality in \eqref{E-QDS}. Let $\ld>0$ and $Q_\ld$ as defined in \eqref{E-Qld}. The following are equivalent:
\begin{enumerate} 
\item The semigroup $\sg{T}$ is conservative.
\item The sequence of operators $\{Q_\ld^n(\One)\}_{n \geq 0}$ converges $\sm$-weakly to 0.
\item If for some $x \in \loph$, we have $Q_\ld x = x$, then $x =0$.
\item If for some $x \in \loph$, we have $Q_\ld x \geq x$, then $x =0$.
\item The operator $Q_{\ld *}$ is mean ergodic, where $Q_{\ld *}$ denotes the predual operator of $Q_\ld$.
\item $\lim_{n \to \infty}\norm{Q_{\ld *}^n \rho}_{tr} =0$ for all $\rho \in \Tsa_+$.
%\item For each $\rho \in \Tsa_+$, the sequence  $\{Q_{\ld *}^n\rho\}_{n \geq 0}$ is weakly compact with $0$ as a weak cluster point.
\end{enumerate}
\end{prop}
Some of the conditions in Proposition \ref{QDScons} are known. For example, conditions (ii) and (iii) in Proposition \ref{QDScons} are similar to the conditions given in \cite[Theorem 3.2]{F99}. It should be noted however, that \cite[Theorem 3.2]{F99} was proven without applying honesty theory.

We now generalise the notion of conservativity to the class of minimal quantum dynamical semigroups constructed in Proposition \ref{QDSGen} by transfering the concept of honesty from Kato's Theorem to these quantum dynamical semigroups:
\begin{defn}
Let $\sg{T}$ be the minimal quantum dynamical semigroup identified in Proposition \ref{QDSGen} with $\Ups$ satisfying Premise \ref{QDShyp}. The semigroup $\sg{T}$ is said to be honest if and only if its predual semigroup
$\sg{S}$ is honest in the sense of Definition \ref{Def-HAss}.
\end{defn}
We will show that honesty is the natural analogue of conservativity in the strictly substochastic case. As we mentioned above, conservativity is important because it allows us to characterise uniqueness of the semigroup and also the domain of its generator. It turns out that honesty also allows us to do the same for the substochastic case as we will show in Corollary \ref{HonDom} and Proposition \ref{UniqQDSHon}.

As in the conservative case, we are also interested in characterising the honesty  of the quantum dynamical semigroups. Since Proposition \ref{QDScons} was derived from characterisations of honesty in Theorems \ref{HT-R}, \ref{HT-ME}, \ref{HT-dual}, it follows that conditions (ii) to (v) in Proposition \ref{QDScons} also characterise the honesty of the minimal semigroup $\sg{T}$. These conditions can also be used to show that other previously known characterisations of conservativity (which were proven without using honesty theory methods) also characterise honesty. 
%consider the following characterisation of conservativity based on the map $\Ups: \loph \to$ Sesquilinear Forms from \cite{F99}.
%\begin{prop}\label{ConsUps}\cite[Proposition 3.3.1]{F99}
%Let $\sg{T}$ be the minimal quantum dynamical semigroup identified in Proposition \ref{QDSGen} with $\Ups$ satisfying Premise \ref{QDShyp} with equality in \eqref{E-QDS}. Then $\sg{T}$ is conservative if and only if $\ker(\ld-\Ups) = \{0\}$ for some/all $\ld>0$.
%\end{prop}   
For example, we can prove a version of \cite[Proposition 3.3.1]{F99} for honesty by applying condition (iii) of Proposition \ref{QDScons}. We simply require the following lemma whose proof we omit as it can be proven almost exactly as in \cite[Proposition 3.30]{F99}.% which describes the relationship between the eigenvalues of the map $\Ups: \loph \to$ Sesquilinear Forms and the operator $Q_\ld, \ld>0$ defined in \eqref{E-Qld}. 
\begin{lem}\label{kerUps} Fix $\ld>0$. Then for all $x \in \loph$, we have $\Ups(x) = \ld x$ if and only if
$Q_\ld(x) = x$.
\end{lem}
%This result combined with Proposition \ref{QDScons} (iii) then allows us to deduce the analogue of \cite[Proposition 3.3.1]{F99} for the case of honesty.
\begin{cor}\label{HT-form} Let $\sg{T}$ be the minimal quantum dynamical semigroup identified in Proposition \ref{QDSGen} with $\Ups$ satisfying Premise \ref{QDShyp}. The semigroup $\sg{T}$ is honest if and only if $\ker(\ld-\Ups) = \{0\}$ for some/all $\ld>0$.
\end{cor}

%Lemma \ref{kerUps} also allows us to derive an alternative description of $\ker(\ld-\Ups)$. Since we have $Q_\ld = (\BR)^*$ (Lemma \ref{PQBR}) and we know that $\ker(I-(\BR)^*) = \ker(\ld-(A+B)^*)$, we have as a corollary:
%\begin{cor}
%$\ker(\ld-\Ups) = \ker(\ld-(A+B)^*)$ for all $\ld>0$. 
%\end{cor}

Now let us look at the domains of generators of these minimal semigroups. We will see that honesty theory allows us to give two different descriptions of the domain of the generator of an honest semigroup, one in terms of cores (Proposition \ref{Vhon}) and the other, a description of the actual domain (Proposition \ref{HonDomF}). The description in terms of cores given below is an extension of \cite[Proposition 3.32]{F99} which states that $\V$ is a core for $G$ if and only if $\sg{T}$ is conservative.
\begin{prop}\label{Vhon} The minimal quantum dynamical semigroup $\sg{T}$ is honest if and only if any of the spaces $\V_n$ or $\pi^n(\Tf), n \in \N$ is a core for $G$, the generator of the predual semigroup $\sg{S}$. 
\end{prop}
\begin{proof} Theorem \ref{HT-R} tells us that $\sg{T}$ is honest if and only if $G = \ov{A+B}$. Since $B$ is $A$-bounded, a subspace $\mathcal{S} \subset D(A)$ is a core for $\ov{A+B}$ if and only if $\mathcal{S}$ is a core for $A$. The result follows from Lemma \ref{r1coreA} and Remark \ref{PiCore} which tell us that for each $n \in \N$, $\V_n$ and $\pi^n(\Tf)$ are cores for $A$.
\end{proof}

Next, we will give a precise description of the domain of the generator of an honest semigroup. But first, we need some auxiliary information. Recall from \eqref{E-QDSGen} that $\Ups(x)$ is a sequilinear form on $D(Y) \x D(Y)$ for all $x \in \loph$. So if $\Ups(x)$ is closed, we can associate an operator $W(x)$ with the form $\Ups(x)$ in the sense: \begin{align}\label{FormOp}
D(W(x)) &=\{ u \in D(\Ups(x)):\exists v \in \Hb \text{ such that }\Ups(x)(u, \phi) = \inp{v, \phi} \text{ for all }\phi \in D(\Ups(x)) \},\notag\\
W(x)u &= v.
\end{align} Define 
\begin{align*}\mathcal{F}:=\{x \in \loph:\,& \text{there exists }W(x) \in \loph\text{ such that }\\&\Ups(x)[v,u] = \inp{v,W(x)u} \text{ for all }u,v \in D(Y)\}.
\end{align*}  

\begin{lem}\label{ABUps} $\mathcal{F} = D((A+B)^*)$ and $W(x) = (A+B)^*x$ for all $x \in \mathcal{F}$.
\end{lem}
\begin{proof} We begin by showing that $W \supseteq (A+B)^*$. Let $x \in D((A+B)^*)$, $u,v \in D(Y)$. Then by Lemma \ref{ABr1}, \[\inp{v, ((A+B)^*x)u} =
\Tr(((A+B)^*x) \ruv)= \Tr(x ((A+B)\ruv)) = \Ups(x)[v,u].\] Thus $\Ups(x)$ is given by a bounded operator, namely $(A+B)^*x$ so $x \in \mathcal{F}$ and $W \supseteq (A+B)^*$.

Now let $x \in \mathcal{F}$ and $u,v \in D(Y)$. Then by Lemma \ref{ABr1}, \[\Tr(W(x) \ruv) = \inp{v,W(x) u} = \Ups(x)[v,u] = \Tr(x((A+B)\ruv)).\] So there exists $y=W(x) \in \loph$ such that $\Tr(x ((A+B)\ruv))=\Tr(y\ruv)$ for all $u,v \in D(Y)$. Since by Lemma \ref{r1coreA} we know that $\mathcal{V}$ is a core for $A+B$, it follows that $\Tr(x ((A+B)\rho))=\Tr(y\rho)$ for all $\rho \in D(A+B)$. Therefore $x \in D((A+B)^*)$ and $W(x)=y = (A+B)^*x$. So $W \subseteq (A+B)^*$.
\end{proof}

Lemma \ref{ABUps} allows us to give the following precise description of the domain of the generator when the semigroup is honest because by Theorem \ref{HT-R}, the semigroup is honest if and only if $G=\ov{A+B}$, i.e. if and only if  $G^* = (A+B)^*=W$. 
\begin{prop}\label{HonDomF} The minimal quantum dynamical semigroup $\sg{T}$ is honest if and only if $D(G^*) = \mathcal{F}$ and $G^*x=W(x)$ for all $x \in D(G^*)$.
\end{prop}

As a corollary, we have a characterisation of an honest quantum dynamical semigroup in terms of the form $\Ups(x)$. A similar result was proven for the case of conservativity in \cite[Proposition 3.33]{F99} and \cite[Theorem 4.1]{GQ} but with different proofs as they did not apply honesty theory results.
\begin{cor}\label{HonDom} The minimal quantum dynamical semigroup $\sg{T}$ is honest if and only if the domain of its generator $G^*$ is the space of all elements $x \in \loph$
such that the form $\Ups(x)$ on $D(Y) \x D(Y), (v,u) \mapsto \Ups(x)[v,u]$ is norm continuous.
\end{cor}
\begin{proof} The form $\Ups(x)$ on $D(Y) \x D(Y)$ is norm continuous if and only if there exists an operator $W(x) \in \loph$ such that $\Ups(x)[v,u] = \inp{v,W(x)u}$ for all $u,v \in D(Y)$. So the set of all elements $x \in \loph$ such that the form $\Ups(x)$ is norm continuous is precisely $\mathcal{F}$. The result now follows from Proposition \ref{HonDomF} .
\end{proof}

Lemma \ref{ABUps} also allows us to show that honesty characterises uniqueness of semigroups satisfying \eqref{E-QDSGenT}. This result is an extension of \cite[Corollary 3.23]{F99}, which tells us that a minimal quantum dynamical semigroup which is conservative is unique.
\begin{prop}\label{UniqQDSHon} The minimal quantum dynamical semigroup $\sg{T}$ is honest if and only if it is the
unique contractive quantum dynamical semigroup on $\loph$ satisfying equation \eqref{E-QDSGenT}.
\end{prop}
To prove Proposition \ref{UniqQDSHon} using honesty theory, we require a result describing the relationship between the generators of quantum dynamical semigroups satisfying \eqref{E-QDSGenT} and the operators $A,B$ in Kato's Theorem. 

\begin{lem}\label{QDSGAB} Suppose $\sg{\Tc}$ is a contractive quantum dynamical semigroup on $\loph$ with generator $\G^*$. Then $\sg{\Tc}$ satisfies \eqref{E-QDSGenT} for all $u,v \in D(Y)$ if and only if $\G^* \subseteq (A+B)^*$.
\end{lem}
\begin{proof} By Proposition \ref{GenForm}, $\sg{\Tc}$ satisfies \eqref{E-QDSGenT} if and only if its generator $\G^*$ satisfies \eqref{E-LinbGen}. Since $\G^*x \in \loph$ for all $x \in D(\G^*)$, it follows from Lemma \ref{ABUps} that \eqref{E-LinbGen} holds if and only if $D(\G^*) \subseteq D((A+B)^*)$ and $\G^*x = (A+B)^*x$ for all $x \in D(\G^*)$.
\end{proof}

\begin{proof}[Proof of Proposition \ref{UniqQDSHon}] We begin by noting that if $S, T$ are closed and densely defined operators, then it follows from \cite[Proposition
B.10]{ABHN} that $S \subseteq T$ if and only if $T^* \subseteq S^*$. Since generators of $C_0$-semigroups are closed and densely defined, and $(A+B)^* = \ov{A+B}^*$, it follows that if $\G$ is a generator of a $C_0$-semigroup, then $\G^* \subseteq
(A+B)^*$ if and only if $A+B \subseteq \G$ . 

Now recall that we denote the generator of $\sg{T}$ by $G^*$ and suppose that there is another quantum dynamical semigroup $\sg{\Tc}$ satisfying \eqref{E-QDSGenT} with generator denoted $\G^*$. Then by Lemma \ref{QDSGAB},
both $\G^*, G^*\subseteq (A+B)^*$ and thus $A+B \subseteq \G$ and $A+B \subseteq G$. So there exist at least two substochastic predual semigroups which have generators that are extensions of
$A+B$. By Theorem \ref{H-UAB}, this occurs if and only if the minimal semigroup is dishonest. Therefore $\sg{T}$ is unique if and only if it is honest.
\end{proof}

We will conclude this section by giving an application of honesty theory to a strictly substochastic quantum dynamical semigroup. As honesty theory for the strictly substochastic case has yet to be studied in the literature, we will modify an example from \cite[Section 2]{FM} (which was used for the study of conservativity) to form a strictly substochastic semigroup. Our modification is motivated by the addition of a potential or absorption term to the generator of a substochastic semigroup in classical $L^1$ examples (see for instance, Example \ref{EgCTransp}).
\begin{eg}\label{SsQDS} We begin by looking at the example from \cite{FM}. Let $\Hb:= L^2(\R,\C)$, the space of complex-valued square-integrable functions on the real line and let the operators $Y, (L_l)_{l \in \N}$ be defined as: 
\begin{align}\label{E-SubstExQDS}
(Yu)(x)&=\frac{1}{2}\sm(x)^2 u''(x),\quad D(Y)=\{u \in \Hb\,:\,u', u'' \in \Hb\},\notag\\
(Lu)(x)&=L_1u(x)= \sm(x)u'(x),\quad D(L)=\{u \in \Hb\,:\,u' \in \Hb\},\notag\\
&\text{where }\sm(x) \text{ is a complex-valued function defined on }\R,\\
(L_lu)(x)&=0, \quad l \geq 2.\notag
\end{align} 
For simplicity, we will consider two cases, namely $\sm(x) = 1$ and $\sm(x) = -ie^{ix}$. Note that  if $\sm(x)=1$, then $Y$ is simply the Laplacian on $\R$. It can be shown that (some realisation of) $Y$ is a self-adjoint operator which generates a substochastic semigroup on $\Hb$ and moreover, these operators satisfy Premise \ref{QDShyp} with equality (see \cite{FM}). More importantly, 
\begin{prop}\label{EgUnmod}\cite[Theorem 2.1, Remark 3.5]{FM} The semigroup is honest if $\sm(x)=1$ and dishonest
if $\sm(x) = -ie^{ix}$ . 
\end{prop}
More generally, if $\sm$ is a real-valued bounded smooth function in $\R$ with bounded derivatives of all orders or if $\sm$ is multiplied by a complex phase independent of $x$, then the minimal quantum dynamical semigroup constructed from $Y$ and $L$ above is conservative (or honest) \cite[Remark 3.5]{FM}. Such semigroups occur in the dilation and quantum extension of classical diffusion processes on $\R$.  

We use a simple modification in order to obtain the strictly substochastic example. First consider the Schr\"{o}dinger operator, defined formally by $S_K:=\frac{1}{2}\Delta-K$ for some measurable function $K$. This operator is related to diffusion processes with absorption. Under some additional conditions, the Schr\"{o}dinger operator is a self-adjoint operator generating a substochastic semigroup in $\Hb$ (see \cite{V86} for example). We will
consider the case when $K$ is a strictly positive, bounded, real-valued function and we will also let $K$ denote the operator of multiplication with this function. We now define
\[(Y_Ku)(x):=\frac{1}{2}\sm(x)^2 u''(x)-K(x)u(x),\quad D(Y_K)=\{u \in \Hb\,:\,u', u'' \in \Hb\}\]
and leave $(L_l)_{l \in \N}$ as defined in \eqref{E-SubstExQDS}. 

Once again, we consider the two cases above, namely $\sm(x) = 1$ and $\sm(x)=-ie^{ix}$. Since $K$ is multiplication with a strictly positive function, $Y_K$ and $(L_l)_{l \in \N}$ satisfy \eqref{E-QDS} with a strict inequality. Hence the minimal quantum dynamical semigroup associated with $Y_K,
(L_l)_{l \in \N}$ is strictly substochastic. 

From Section \ref{S-QDSKato}, we know that the formal operators 
\begin{align*}
A_K\rho&:=Y_K\rho+\rho Y_K^*=Y\rho+\rho Y^*- (K\rho+\rho K), \quad \rho \in D(A) \text{ and }\\
B\rho &:= L_1 \rho L_1, \quad \rho \in D(B)
\end{align*}
satisfy Kato's Theorem. Since the operator $K$ is bounded, so is the operator $\K:\Tf \to \Tf$, $\K\rho:= K\rho+\rho K$. Hence we may apply our results relating potentials and honesty from Section \ref{S-HonPot} to derive some results on the honesty of this semigroup. In particular, combining Proposition \ref{HonPot} and Proposition \ref{DishonPot} with Proposition \ref{EgUnmod}, it follows that
\begin{prop}
The minimal quantum dynamical semigroup associated with $Y_K$ and $L$ is honest for the case $\sm(x)=1$ and dishonest for the case $\sm(x)=-ie^{ix}$.  
\end{prop}
Therefore, if $\sm(x)=1$ (or more generally, $\sm$ is a real-valued, bounded, smooth function in $\R$ with bounded derivatives of all orders), then the minimal
semigroup is the unique semigroup satisfying \eqref{E-QDSGenT} (Proposition \ref{UniqQDSHon}). Moreover, by Corollary \ref{HonDom}, in this case we have a precise description of the domain of the generator of the minimal semigroup, namely $D(G^*)$ is given by all elements $x \in \loph$ such that the form $\Ups(x)$ on $D(Y) \x D(Y), (v,u) \mapsto \Ups(x)[v,u]$ is norm
continuous. 
\end{eg}

\bibliographystyle{plain}	% (uses file "plain.bst")
\bibliography{ref}

%\begin{thebibliography}{9}
%\end{thebibliography}

\end{document}